\documentclass [10pt]{article}

\usepackage{geometry}
\usepackage{amssymb}
\usepackage{amsfonts}
\usepackage{graphicx}
\usepackage{amsmath}

\usepackage{hyperref}
\hypersetup{colorlinks,
            citecolor=red, 
            filecolor=black,
            linkcolor=black,
            urlcolor=black}

\newtheorem{theorem}{Theorem}[section]

\newtheorem{assumption}[theorem]{Assumption}

\newtheorem{corollary}[theorem]{Corollary}

\newtheorem{definition}[theorem]{Definition}

\newtheorem{lemma}[theorem]{Lemma}

\newtheorem{proposition}[theorem]{Proposition}
\newtheorem{remark}[theorem]{Remark}

\newenvironment{proof}[1][Proof]{\noindent\textit{#1.} }{\hfill \rule{0.5em}{0.5em}}

\newcommand{\R}{\mathbb{R}}
\renewcommand{\eqref}[1]{(\ref{#1})}
\renewcommand{\d}{{\rm d}}

\usepackage{xcolor}

\begin{document}

\title{\textbf{An integrated semigroup approach for age structured equations with diffusion and non-homogeneous boundary conditions}}
\author{\textsc{Arnaud Ducrot$^{a}$, Pierre Magal$^b$ and Alexandre Thorel$^a$}\\
{\small \textit{$^{a}$ Normandie Univ, UNIHAVRE, LMAH, FR-CNRS-3335, ISCN, 76600 Le Havre, France}}\\
		{\small \textit{email: arnaud.ducrot@univ-lehavre.fr}}\\
		{\small \textit{email: alexandre.thorel@univ-lehavre.fr}}\\
	{\small \textit{$^{b}$ Univ. Bordeaux, IMB, UMR 5251, F-33400 Talence, France}} \\
		{\small \textit{CNRS, IMB, UMR 5251, F-33400 Talence, France.}}\\
		{\small \textit{email: pierre.magal@u-bordeaux.fr}} 
\\
}
\maketitle
\begin{abstract}

In this work, we consider a linear age-structured problem with diffusion and non-homogeneous boundary conditions both for the age and the space variables. We handle this linear problem by re-writing it as a non-densely defined abstract Cauchy problem. To that aim we develop a new result on the closedness of a commutative sum of two non-densely defined operators by using the theory of integrated semigroups. As an application of this abstract result, we are able to associate a suitable integrated semigroup to some age-structured problem with spatial diffusion and equipped with non-homogeneous boundary conditions. This integrated semigroup is characterized by the description of its infinitesimal generator.
Further applications of our abstract result are also given to the commutative sum of two almost sectorial operators, for which we derive a closedness results.

\vspace{0.2in}\noindent \textbf{Key words}. Age-structured problem with diffusion, non densely-defined operators, almost sectorial operators, integrated semigroups, commutative  sum of linear operators.

\vspace{0.1in}\noindent \textbf{2010 Mathematical Subject Classification}. 47A10, 47D06, 47D62, 47F05, 47D03. 

\end{abstract}

\section{Introduction}

The study of linear non-homogeneous abstract Cauchy problems involving non-densely defined operator has a long history starting with the pioneer work of Da Patro and Sinestrari \cite{DaPrato-Sinestrari}. 
 The theory related to such equations has received a lot of interests in the last decades with a huge development of the so-called integrated semigroups. We refer the reader to  \cite{Arendt1, Arendt2, Arendt-book, Kellermann, Magal-Ruan07, Magal-Ruan09b, Neubrander, Thieme90a}, to the recent monograph of Magal and Ruan \cite{Magal-Ruan17} and the references therein cited. Indeed linear and semilinear abstract Cauchy problems with non-densely defined operators allow to handle various different classes of equations such as delay differential equations, age structured equations, parabolic equations with nonlinear and nonlocal boundary conditions and others (see \cite{Magal-Ruan17} for an overview of some possible applications of integrated semigroups to various types of equations). Formulation of partial differential equations as an abstract Cauchy problem allows on the one hand to deal with the existence and the uniqueness of mild solutions, and more interestingly to deal with various perturbation results and spectral theory. On the other hand, this also provides a powerful tool to study semilinear problems by using the constant variation formula. It allows in particular to study stability and instability properties, to construct smooth invariant manifolds and also to deal with bifurcation theory. We refer to the reader to \cite{DM-TAMS, LMR12, Magal-Ruan09a} and references therein cited. See also \cite{MLR14} for normal form reduction for semilinear abstract Cauchy problems, involving non-densely defined operators.

In this work we study the following linear age-structured problem with diffusion and non-homogeneous boundary conditions
\begin{equation}\label{eq-lin}
\begin{cases}
\left(\partial_t +\partial_a-\Delta_x\right)u(t,a,x)=f(t,a,x), & t>0,\;a>0,\;x\in\Omega,\\
u(t,0,x)=g(t,x), & t>0,\;x\in\Omega,\\
\nabla u(t,a,x)\cdot \nu(x)=h(t,a,x), & t>0,\;a>0,\;x\in\partial\Omega, \\
u(0,a,x) = u_0(a,x), & a>0,\;x \in \Omega.
\end{cases}
\end{equation}
Herein $\Omega\subset \R^N$ denotes a bounded spatial domain with a smooth boundary $\partial\Omega$, $\nu(x)$ with $x\in\partial \Omega$ denotes the outwards unit vector to $\partial\Omega$, while $f=f(t,a,x)$, $u_0(a,x)$, $g=g(t,x)$ and $h=h(t,a,x)$ denote suitable functions those regularities and specific properties will be described latter. The system \eqref{eq-lin}  do not contain any mortality terms and birth terms since it can be considered from our analysis and by using bounded perturbation technics  (similar to the ones presented in \cite{Magal-Ruan17}).

Linear as well as non-linear age-structured equations with diffusion has been introduced in various contexts, in particular in mathematical biology to model the interplay between the spatial motions of populations and the history of individuals, that somehow generalizes diffusion equations with time delay. We refer to the reader to the pioneer works of Gurtin and MacCamy \cite{Gurtin1, Gurtin2}. We also refer to the more recent developments in \cite{DiBlasio, %Ducrot,
 Ducrot-Magal, DM, DM-TAMS, Kubo-Langlais, Langlais, %Walker,
  Walker1, Webb08} and the references therein cited for various studies of (linear or non-linear) age structured equations with spatial diffusion and more generally for analysis of problems coupling hyperbolic and diffusive effects.\\
Here we aim at considering non-homogeneous perturbations both at the spatial boundary $\partial\Omega$ and at the age boundary $a=0$. To our best knowledge this paper is the first work dealing such a boundary conditions. And as a special case, we able to provide the form of the mild solutions for such a problem (see Lemma \ref{LEmild}). Here we consider a simple age structured model, but the analysis (and the mild solutions) obtained in the present article can be extended in several ways to handle more realistic models, as for instance by taking into account mortality rates. 
The study of such a problem can be handled using different strategies. In view of the recent literature, to consider such a question, one may split the above problem into three sub-problems separately taking into account the initial data at $t=0$, the non-homogeneous boundary condition at $a=0$ and the non-homogeneous Neumann boundary condition at $\partial\Omega$.
In other words, a solution of \eqref{eq-lin} can be considered as $u=u_1+u_2+u_3$ wherein $u_1$, $u_2$ and $u_3$ respectively solve the following sub-problems
\begin{equation}\label{eq-lin1}
\begin{cases}
\left(\partial_t +\partial_a-\Delta_x\right)u_1(t,a,x)=f(t,a,x), & t>0,\;a>0,\;x\in\Omega,\\
u_1(t,0,x)=0, & t>0,\;x\in\Omega,\\
\nabla u_1(t,a,x)\cdot \nu(x)=0, & t>0,\;a>0,\;x\in\partial\Omega,\\
u_1(0,a,x)=u_0(a,x), & a>0,\;x \in \Omega,
\end{cases}
\end{equation}
 \begin{equation}\label{eq-lin2}
\begin{cases}
\left(\partial_t +\partial_a-\Delta_x\right)u_2(t,a,x)=0, & t>0,\;a>0,\;x\in\Omega,\\
u_2(t,0,x)=g(t,x), & t>0,\;x\in\Omega,\\
\nabla u_2(t,a,x)\cdot \nu(x)=0, & t>0,\;a>0,\;x\in\partial\Omega,\\
u_2(0,a,x)=0, & a>0,\;x \in \Omega,
\end{cases}
\end{equation}
and
\begin{equation}\label{eq-lin3}
\begin{cases}
\left(\partial_t +\partial_a-\Delta_x\right)u_3(t,a,x) = 0, & t>0,\;a>0,\;x\in\Omega,\\
u_3(t,0,x) = 0, & t>0,\;x\in\Omega,\\
\nabla u_3(t,a,x)\cdot \nu(x) = h(t,a,x), & t>0,\;a>0,\;x \in \partial\Omega,\\
u_3(0,a,x) = 0, & a>0,\;x \in \Omega.
\end{cases}
\end{equation}

Each of these sub-problems can be re-written as abstract Cauchy problems when considered in suitable Lebesgue spaces. The first one can be handled by the commutative sum of two densely defined operators (see \cite{Daprato-Grisvard}) while the second and the third can be handled using more recent results (see \cite{DM-TAMS, Thieme08}) on the commutative sum of one densely defined operator and one non-densely defined operator. More details on this approach will be given in Section 4. This interesting approach allows us to deal with the existence of solutions for \eqref{eq-lin} but does not directly provide a reformation of the full problem \eqref{eq-lin} as an abstract Cauchy problem. In this work we develop another methodology to study \eqref{eq-lin} by deriving a new and rather general result on the closedness of a commutative sum of two non-densely defined operators. As a special case, one can apply our general result to problem \eqref{eq-lin} by re-writing it as a linear non-densely defined abstract Cauchy problem, that generates a suitable integrated semigroup and for which we able to describe its infinitesimal generator. This direct approach allows us to deal with the existence and uniqueness of mild solutions for such a linear problem but it also provides a tractable analytical framework to handle related semilinear problem involving, for instance, nonlinear and nonlocal boundary conditions both at age $a=0$ and on the spatial boundary $\partial\Omega$. This non-linear issue will not be discussed in this paper (see \cite{Magal-Ruan17} for non-linear studies related to integrated semigroups). Furthermore, as a by-product of our general analysis, we also obtains a new abstract result ensuring that the closure of the sum of two resolvent commuting non-densely defined operators becomes an almost sectorial operator (see \eqref{DEF-AS} below for the definition of such operators).  

In the literature the properties of the sum of two resolvent commuting linear operators has been extensively studied, starting from the pioneer work of Da Prato and Grisvard \cite{Daprato-Grisvard}. One may for instance refer to \cite{arendt-bu-haase, dore-venni,% monniaux,
 pruss-sohr%, pruss-sohr2
 } where closedness, sectorial property or boundedness of imaginary powers of the commutative sums of two operators have been investigated, under various assumptions and mostly for densely defined operators. One may also refer to \cite{%Kalton, 
labbas, roidos, roidos2} where the closability and the invertibility of the commutative sum of two sectorial operators has been discussed using functional calculus, allowing possibly operators with non-dense domains. Here we consider different settings where the functional calculus used in the aforementioned works does not apply in general. Rather our arguments and our analysis are based on the integrated semigroups theory, in the spirit of the works of Thieme \cite{Thieme08} and Ducrot and Magal \cite{DM-TAMS}.

This paper is organized as follows. In Section 2 we first recall important results on integrated semigroups, that will be used in throughout this work. We also state our main assumptions and our results on the closedness and some properties of the commutative sum of two non-densely defined linear operators. Section 3 is concerned with the proof of the main result while Section 4 is devoted to the  application of our abstract result to problem \eqref{eq-lin}.

\section{Definitions, assumptions and main results}

In this section we state the main results of this note, that are firstly concerned with the closedness of the resolvent commuting sum of two non-densely defined linear operators. Secondly we discuss some conditions ensuring that the closure is the infinitesimal generator of an integrated semigroup. Throughout this section, we denote by $\left(X,\|\cdot\|\right)$ a given and fixed Banach space.\\
Before going to our main results let us recall some definitions and discuss the main assumptions that will be used in this work.

\subsection{Preliminary and definitions}

In this subsection we present some materials on linear abstract Cauchy problems involving non-densely defined operators and we recall some important definitions and results about integrated semigroups that will be used throughout this work. We refer the reader to the monograph \cite{Magal-Ruan17}, the references therein cited and the references cited above for more details and more results on this topic.\\
In the following, when $X$ and $Z$ are two given Banach spaces we denote by $\mathcal{L}\left( X,Z\right)$ the Banach space of the bounded linear operators from $X$ into $Z$ and we denote for simplicity by $\mathcal{L}\left( X\right) $ the space $\mathcal{L}\left( X,X\right).$

\begin{definition}
Let $T:D(T)\subset X\to X$ be a linear operator. We denote by 
$${\rm Graph}\,(T):=\left\{ (x,Tx),\;x\in D(T)\right\}\subset X\times X,$$
its graph. We say that the linear operator $T$ is closable if it admits a closed extension, that is there exists a closed linear operator $T':D(T')\subset X\to X$ such that
\begin{equation*}
\overline{{\rm Graph}\,(T)}\subset {\rm Graph}\,(T').
\end{equation*}
In that case, we define the closure of $T$, denoted by $\overline{T}:D\left(\overline{T}\right)\subset X\to X$, as the smallest closed extension of $T$, namely
\begin{equation*}
D(\overline T)=\{x\in X:\;\exists ! y\in X \;(x,y)\in \overline{{\rm Graph}\,(T)}\}\text{ and }\overline Tx=y,\quad \forall x\in D(\overline T).
\end{equation*}
In other words, one has $x\in D(\overline T)$ and $y=\overline Tx$ if and only if there exists a sequence $(x_n)_{n\geq 0}\subset D(T)$ such that $\|x_n-x\|\to 0$ and $\|Tx_n-y\|\to 0$ as $n\to \infty$.
\end{definition}
Throughout the rest of this paper, if $T:D(T)\subset X\to X$ is a linear operator, we use the notation ${\rm R}(T)$ to denote the range of $T$, that is ${\rm R}(T)=T \left( D(T) \right) \subset X$.

In the following we first recall some known results on integrated semigroups and then we recall some important results on integrated semigroups generated by the so-called almost sectorial operators.

\subsubsection{Integrated semigroups}

In this section we consider a linear operator $A:D(A)\subset X\to X$ and we denote by $\rho(A)$, the resolvent set of $A$. We assume that $(A,D(A))$ that satisfies the following set of conditions.
\begin{assumption}\label{ASS1}
The linear operator $A:D(A)\subset X\to X$ satisfies the following {\bf weak Hille-Yosida} property 
\begin{itemize}
\item[{\rm (i)}] There exist $\omega_A\in\R$ and $M_A\geq 1$ such that $(\omega_A,\infty)\subset \rho(A)$ and
\begin{equation*}
\|(\lambda-A)^{-k}\|_{\mathcal L\left(\overline{D(A)}\right)}\leq \frac{M_A}{(\lambda-\omega_A)^k},\quad\forall \lambda>\omega_A,\;k\geq 1.
\end{equation*}
\item[{\rm (ii)}] For all $x\in X$, one has $(\lambda-A)^{-1}x\to 0$ as $\lambda\to\infty$.
\end{itemize}
\end{assumption}
We denote by $A_0:D(A_0)\subset\overline{D(A)}\to \overline{D(A)}$ the part of $A$ in $\overline{D(A)}$ the linear operator given by
\begin{equation}\label{Def A0}
D(A_0)=\left\{x\in D(A):\;Ax\in \overline{D(A)}\right\}\text{ and }A_0 x=Ax,\quad\forall x\in D(A_0).
\end{equation}
Let us first recall that, since $\rho(A)\neq \emptyset$, one has $\rho(A)=\rho(A_0)$ (see Magal and Ruan \cite[Lemma 2.1]{Magal-Ruan09a}) and that, under Assumption \ref{ASS1} ${\rm (i)}$, $A_0$ is the infinitesimal generator of a strongly continuous semigroup on $\overline{D(A)}$, denoted by $\left\{T_{A_0}(t)\right\}_{t\geq 0}$. Moreover it satisfies the following estimate
\begin{equation*}
\|T_{A_0}(t)\|_{\mathcal L(\overline{D(A)})}\leq M_A e^{\omega_A t},\quad \forall t\geq 0.
\end{equation*}
We shall denote by $\omega _{0}(A_{0})\in [-\infty,\infty)$ its the growth rate.
%
%
%Now recall that $\omega _{0}(A_{0})$, the growth rate of the semigroup $\left\{
%T_{A_{0}}(t)\right\} _{t\geq 0}$, is defined by 
%\begin{equation*}
%\omega _{0}\left(A_{0}\right):=\lim_{t\rightarrow \infty }\frac{\ln \left( \left\Vert
%T_{A_{0}}(t)\right\Vert _{\mathcal{L}\left( \overline{D(A)}\right) }\right) 
%}{t}.
%\end{equation*}
And since $\rho(A)=\rho(A_0)$ this in particular yields $
\left( \omega _{0}(A_{0}),\infty \right) \subset \rho(A).$\\
Next, according to \cite{Magal-Ruan07}, under the above assumption, $A$ is the generator of a unique non degenerate integrated semigroup, denoted by $\left\{S_A(t)\right\}_{t\geq 0}\mathcal L\left(X,\overline{D(A)}\right)$, that is defined for each $x\in X$, $t\geq 0$ and $\mu>\omega_A$ by
$$
S_A(t)x=\mu \int_0^t T_{A_0}(s)(\mu-A)^{-1}x\d s+(\mu-A)^{-1}x-T_{A_0}(t)(\mu-A)^{-1}x,
$$
and that furthermore satisfies
\begin{itemize}
\item[{\rm (i)}] For each $x\in X$ and $t\geq 0$,
\begin{equation*}
\int_0^t S_A(s)\d s\in D(A) \quad \text{and} \quad S_A(t)x=A\int_0^t S_A(s)\d s+tx.
\end{equation*}
\item[{\rm (ii)}] For each $t,s\geq 0$ one has $S_A(t+s)-S_A(s)=T_{A_0}(s)S_A(t)$.
\end{itemize}
We continue this section by recalling some results about the constant variation formula. To that aim we now fix $\tau>0$ and for each $f\in L^1(0,\tau;X)$ we define the convolution $S_A\ast f$ by
\begin{equation*}
\left(S_A\ast f\right)(t)=\int_0^t S_A(s)f(t-s)\d s, \quad \forall t\in [0,\tau].
\end{equation*}
Note that when $f\in C^1([0,\tau];X)$ the convolution map $t\mapsto (S_A\ast f)(t)$ is continuously differentiable from $[0,\tau]$ into $\overline{D(A)}$.

Next, the existence and uniqueness of mild solutions for the non-homogeneous Cauchy problem 
\begin{equation}\label{CP}
\displaystyle \frac{\d u(t)}{\d t}=A u(t)+f(t),\quad t>0\text{ with }u(0)=x\in\overline{D(A)},
\end{equation}
for some $f\in L_{\rm loc}^p([0,\infty);X)$ and $p\geq 1$, is related to the following condition.
\begin{assumption}\label{ASS2}
There exist $p\in [1,\infty)$, $M>0$ and $\omega\in\R$ such that for all $f\in C^1\left([0,\infty);X\right)$ one has
\begin{equation}\label{estiLp}
\left\Vert  \frac{\d }{\d t} (S_A\ast f)(t)\right\Vert \leq M \left(\int_0^t e^{p\omega(t-s)}\|f(s)\|^p \d s\right)^{\frac{1}{p}},\quad\forall t\geq 0.
\end{equation}
\end{assumption}
\begin{remark}  
For the purpose of this work, it is important to work with $L^p$-estimates with suitable values for $p\in [1,\infty)$. Indeed, as it will be clear in the sequel, such $L^p$-framework will allow us to compensate a singularity coming from almost sectorial perturbation (see \eqref{esti-TB}). 
\end{remark}

Under the above assumptions, namely Assumptions \ref{ASS1} and \ref{ASS2}, the derivation-convolution operator $\frac{\d}{\d t}(S_A\ast\cdot)$ enjoys the following properties. We refer to Magal and Ruan \cite{Magal-Ruan07, Magal-Ruan17} for the proof of the properties stated below and for further properties.
\begin{lemma}\label{LE2.3}
Let Assumptions \ref{ASS1} and \ref{ASS2} be satisfied. For each function $f\in L_{\rm loc}^p([0,\infty);X)$, the convolution function $(S_A\ast f)$ belongs to $C^1\left([0,\infty);\overline{D(A)}\right)$ and the function $u(t):=\frac{\d }{\d t}(S_A\ast f)(t)$ is the unique mild (or integrated) solution of the abstract Cauchy problem \eqref{CP} with $x=0$,
namely the function $t \to u(t)$ satisfies
\begin{equation*}
\int_0^t u(s)\d s\in D(A),\quad \forall t\geq 0 \quad \text{and} \quad u(t)=A\int_0^t u(s)\d s+\int_0^t f(s)\d s,\quad\forall t\geq 0.
\end{equation*}
Moreover, \eqref{estiLp} holds for all $f\in L^p_{\rm loc}\left(\left[0,\infty\right);X\right)$ and, for each $\lambda>\omega_0(A_0)$ one has
$$
\left(\lambda-A_0\right)^{-1}\frac{\d }{\d t}(S_A\ast f)(t)=\int_0^t T_{A_0}(t-s)(\lambda-A)^{-1}f(s)\d s,\quad\forall t\geq 0.
$$
\end{lemma}

In the sequel, for each $\tau>0$ and each $f\in L^1(0,\tau;X)$, we make use of the symbol $\diamond$ to denote the derivation-convolution operator, that is 
\begin{equation*}
\left(S_A\diamond f\right)(t)=\frac{\d }{\d t}(S_A\ast f)(t),\quad\forall t\in [0,\tau],
\end{equation*}
as soon as the map $t\mapsto (S_A\ast f)(t)$ is continuously differentiable on $[0,\tau]$.\\
Using this notation, let us also recall that, for each $\tau>0$ and each $f\in L^p(0,\tau;X)$, the following formula holds for $0\leq s\leq t\leq \tau$
\begin{equation}\label{constant-varf}
\left(S_A\diamond f\right)(t)=T_{A_0}(t-s)\left(S_A\diamond f\right)(s)+\left(S_A\diamond f(\cdot+s)\right)(t-s).
\end{equation}

\subsubsection{Analytic integrated semigroup}

In this subsection we present some materials related to non-densely defined almost sectorial operators. We refer to Ducrot and al. \cite{Ducrot-Magal-Prevost09} and to Ducrot and al. \cite{DM-Pisa} for more details.

Let $B:D(B)\subset X\rightarrow X$ be a closed linear operator.
We denote by $B_{0}$ the part of $B$ in $\overline{D(B)}$. Throughout this section we assume that $B$ satisfies the following assumptions.
\begin{assumption}\label{ASS-analytic} 
Let $B:D(B)\subset X\rightarrow X$ be a linear operator
on a Banach space $\left( X,\left\Vert \cdot\right\Vert \right).$ We assume that
\begin{itemize}
\item[{\rm (a)}] The linear operator $B_{0}$ is the infinitesimal generator of a strongly continuous analytic semigroup of bounded linear operators on $\overline{D(B)}$, denoted by $\left\{ T_{B_{0}}(t)\right\} _{t\geq 0}$.
\item[{\rm (b)}] There exist $\omega_B \in \mathbb{R}$ and $p^{\ast }\in \left[1,\infty \right)$ such that $\left( \omega_B ,\infty \right) \subset \rho\left(B\right) $ and 
\begin{equation*}
\left\Vert \left( \lambda -B\right) ^{-1}\right\Vert _{\mathcal{L}\left( X\right) }=\mathcal O\left(\frac{1}{\lambda ^{\frac{1}{p^{\ast }}}}\right)\text{ as }\lambda\to\infty. 
\end{equation*}
\end{itemize}
\end{assumption}
Let us recall that  Proposition 3.3 in \cite{Ducrot-Magal-Prevost09} ensures that the above conditions are equivalent to $B_0$ is sectorial operator on $\overline{D(B)}$ and $B$ is $\frac{1}{p^*}$-almost sectorial on $X$, in the sense that there exist $\omega\in\R$, $\theta\in\left(\frac{\pi}{2},\pi\right)$ and $M>0$ such that
\begin{equation}\label{DEF-AS} 
\begin{split}
&\Sigma_{\omega,\theta}:=\{\lambda\in\mathbb C\setminus\{\omega\}:\;|{\rm arg}\;(\lambda-\omega)|<\theta\}\subset \rho(B)\text{ and }\\
&|\lambda-\omega|^{\frac{1}{p^*}}\left\|\left(\lambda-B\right)^{-1}\right\|_{\mathcal L(X)}\leq M,\quad\forall \lambda\in \Sigma_{\omega,\theta}.
\end{split}
\end{equation}
We also refer to \cite{Lunardi} for more details. 

Let us observe that, since $p^*<\infty$, the resolvent of $B$, $(\lambda-B)^{-1} \to 0_{\mathcal{L}(X)}$ as $\lambda\to\infty$ in the operator norm topology. Therefore the linear operator $B:D(B)\subset X\to X$ also satisfies Assumption \ref{ASS1} and  the results recalled in the previous subsection hold for the operator $B$. In the following we denote by $\{S_B(t)\}_{t\geq 0}$ and $\left\{T_{B_0}(t)\right\}_{t\geq 0}$ the integrated semigroup and the analytic semigroup generated by $B$ and $B_0$, respectively.\\
Recall now that for any $x\in X$, the map $t\rightarrow S_{B}(t)x$ from $
\left( 0,\infty \right) $ into $X$ is differentiable, so that the family 
\begin{equation*}
T_B(t):=\frac{\d S_{B}(t)}{\d t}=\left( \lambda -B_{0}\right) T_{B_{0}}(t)\left(\lambda -B\right) ^{-1},\quad \text{for }t>0,
\end{equation*}
defines a semigroup of bounded linear operators on $X$. However when $B$ is not densely defined then the family $\left\{ T_B(t)\right\} _{t\geq 0}$ of bounded linear operator on $X$ is not strongly continuous at $t=0$ and the map $t\mapsto T_B(t)$ exhibits a singularity when $t\to 0$.

Next, due to Assumption \ref{ASS-analytic}, the fractional powers operators $\left( \lambda -B_{0}\right) ^{-\alpha }$ are well defined, for any $\lambda >\omega _{0}(B_{0})$ and $\alpha >0$.
%
% by the usual formula
%\begin{equation*}
%\left( \lambda -B_{0}\right) ^{-\alpha }=\frac{1}{\Gamma (\alpha )} \int_{0}^{+\infty }t^{\alpha -1}T_{\left( B_{0}-\lambda \right) }(t)\d t \quad \text{ and } \quad\left( \lambda -B_{0}\right) ^{0}=I.
%\end{equation*}
Since $B$ is only assumed to be $\frac{1}{p^*}$-almost sectorial, the fractional powers of $(\lambda - B)^{-\alpha}$ are only defined $\alpha >0$ large enough, related to $p^\ast$. More precisely, under Assumption \ref{ASS-analytic} and setting 
$$
q^\ast\in (1,\infty] \text{ such that }\frac{1}{p^\ast}+\frac{1}{q^\ast}=1,
$$
for each $\alpha \in \left( \dfrac{1}{q^{\ast }},\infty\right)$ and $\lambda >\omega _{0}(B_{0})$, the fractional power operator $\left( \lambda -B\right) ^{-\alpha }\in \mathcal{L}\left( X, \overline{D(B)}\right)$ is well defined. We refer to cite{periago} or Lemma 3.7 in \cite{Ducrot-Magal-Prevost09} for more details).

%
%More precisely, we have the following result (see \cite{periago} or Lemma 3.7 in \cite{Ducrot-Magal-Prevost09}).
%
%
%
%
%
%\begin{lemma}
%Let Assumption \ref{ASS-analytic} be satisfied. Then, setting $q^\ast\in (1,\infty]$ such that $\frac{1}{p^\ast}+\frac{1}{q^\ast}=1$, for each $\alpha \in \left( \dfrac{1}{q^{\ast }},\infty\right)$ and $\lambda >\omega _{0}(B_{0})$, the fractional power operator $\left( \lambda -B\right) ^{-\alpha }\in \mathcal{L}\left( X\right) $
%is well defined. Moreover, one has 
%\begin{equation*}
%\left( \lambda -B\right) ^{-\alpha }\left( X\right) \subset \overline{D(B)},
%\end{equation*}%
%and the following properties are satisfied:
%\begin{itemize}
%\item[{\rm (i)}] $\left( \mu -B_{0}\right) ^{-1}\left( \lambda -B\right)
%^{-\alpha }=\left( \lambda -B_{0}\right) ^{-\alpha }\left( \mu -B\right)
%^{-1}, \quad \forall \mu >\omega _{0}(B_{0}).$
%
%\item[{\rm (ii)}] $\left( \lambda -B_{0}\right) ^{-\alpha }x=\left( \lambda
%-B\right) ^{-\alpha }x,\quad\forall x\in \overline{D(B)}.$
%
%\item[{\rm (iii)}] For each $\alpha \geq 0,$ $\beta >\dfrac{1}{q^{\ast }}$, we have
%$\left( \lambda -B_{0}\right) ^{-\alpha }\left( \lambda -B\right) ^{-\beta}=\left( \lambda -B\right) ^{-\left( \alpha +\beta \right) }$.
%
%\end{itemize}
%\end{lemma}

These fractional powers allow to derive rather precise properties for the integrated semigroup $\{S_B(t)\}_{t\geq 0}$. In particular we have the following expression for the integrated semigroup $S_B(t)$
\begin{equation*}
S_{B}(t)=\left( \lambda -B_{0}\right) ^{\alpha}\int_{0}^{t}T_{B_{0}}(s)ds\left( \lambda -B\right) ^{-\alpha }\;\forall t\geq 0
\end{equation*}
whenever $\alpha \in \left( \dfrac{1}{q^{\ast }},\infty \right)$ and $\lambda >\omega _{0}(B_{0})$. 
%
%
%
%From this, we infer that for any $\lambda>\omega_0(B_0)$ and $\alpha \in \left( \dfrac{1}{q^{\ast }},1\right]$
%\begin{equation*}
%\frac{\d S_{B}(t)}{\d t}=\left( \lambda -B_{0}\right) ^{\alpha
%}T_{B_{0}}(t)\left( \lambda -B\right) ^{-\alpha },\quad \forall t>0.
%\end{equation*}
As a consequence of the above expression, for $x\in D(B)$, one also obtains, for any $\frac{1}{q^*}<\alpha\leq 1$, that
\begin{equation*}
\frac{\d^2 S_{B}(t)x}{\d t^2}=\left( \lambda -B_{0}\right) ^{\alpha
}T_{B_{0}}(t)B_0(\lambda-B)^{-1}\left( \lambda -B_0\right) ^{-\alpha }(\lambda-B)x,\quad\forall t>0.
\end{equation*}
Hence, fixing $\lambda>\omega_0(B_0)$, one obtains that for each $\omega>\omega_0(B_0)$ and any $\alpha\in \left(\frac{1}{q^*},1\right)$ there exists some constant $M=M_{\alpha,\omega}$ such that, for all $x\in D(B)$, the following important estimate holds
\begin{equation}\label{esti-2nd}
\left\|\frac{\d^2 S_{B}(t)x}{\d t^2}\right\| \leq \frac{M}{t^\alpha}e^{\omega t}\|(\lambda-B)x\|, \quad\forall t>0. 
\end{equation}

%Finally we recall some properties of the non-homogeneous abstract Cauchy problem
%\begin{equation}\label{2.2}
%\begin{cases}
%\displaystyle\frac{\d u(t)}{\d t}=Bu(t)+f(t),\quad t\geq 0,\\
%u(0)=x\in \overline{D(B)}. 
%\end{cases} 
%\end{equation}
We complete this subsection by recalling the following important estimates. 
\begin{theorem}%\label{TH2.3} 
Let Assumption \ref{ASS-analytic} holds. Let $\tau>0$ and $f\in
L^{p}\left( 0,\tau ;X\right) $ with $p>p^{\ast }$ be given. Then, the convolution map 
$t\longmapsto \left( S_{B}\ast f\right) (t)$ 
is continuously differentiable and one has
\begin{equation*}
\left( S_{B}\ast f\right) (t)\in D(B),\quad\forall t\in \left[ 0,\tau \right],
\text{ and }
\left( S_{B}\diamond f\right) (t)=B\int_{0}^{t}\left( S_{B}\diamond f\right)
(s)\d s+\int_{0}^{t}f(s)\d s,\quad\forall t\in \left[ 0,\tau \right] .
\end{equation*}
Moreover, for each $\beta \in \left( \dfrac{1}{q^{\ast }},\dfrac{1}{q}\right)$ (with $\dfrac{1}{q}+\dfrac{1}{p}=1$), each $\lambda >\omega _{0}(B_{0})$ and each $t\in \left[ 0,\tau \right]$, the following equality holds
\begin{equation*}
\left( S_{B}\diamond f\right) (t)=\left(T_B\ast f\right)(t)=\int_{0}^{t}\left( \lambda -B_{0}\right)
^{\beta }T_{B_{0}}(t-s)\left( \lambda -B\right) ^{-\beta }f(s)\d s,
\end{equation*}
as well as the following estimate 
\begin{equation*}
\left\Vert \left( S_{B}\diamond f\right) (t)\right\Vert \leq M_{\beta}\left\Vert \left( \lambda -B\right) ^{-\beta }\right\Vert _{\mathcal{L}(X)}\int_{0}^{t}(t-s)^{-\beta }e^{\omega(t-s)}\left\Vert f(s)\right\Vert \d s,\end{equation*}%
wherein $M_{\beta }$ denotes some positive constant and $\omega>\omega
_{0}(B_{0})$. 
\end{theorem}

\begin{remark}\label{REM-B}
We may note that the above result ensures that the linear operator $B$ satisfies Assumption \ref{ASS2} with any $\omega>\omega_0(B_0)$ and for any $p>p^\ast$.
\end{remark}

%The above result allows us to obtain the following variation of constant formula: 
%\begin{equation*}
%\left( S_{B}\diamond f\right) (t)=T_{B_{0}}(t-s)\left( S_{B}\diamond
%f\right) (s)+\left( S_{B}\diamond f(s+.)\right) (t-s),\quad\forall t\geq s\geq 0.
%\end{equation*}
%Finally, by using the above theorem and the usual uniqueness result of
%Thieme \cite[Theorem 3.7]{Thieme90a}, one derive the following result.
%
%\begin{corollary}%\label{CO2.4} 
%Let Assumption \ref{ASS-analytic} be satisfied. Let $p\in
%\left( p^{\ast },\infty \right) $ be given. Then, for each $f\in
%L^{p}\left( 0,\tau ;X\right) $ and each $x\in \overline{D(B)}$ the abstract
%Cauchy problem \eqref{2.2} has a unique integrated solution $u\in C\left(\left[ 0,\tau \right] ,\overline{D(B)}\right) $ that is given by
%\begin{equation*}
%u(t):=T_{B_{0}}(t)x+\left( S_{B}\diamond f\right) (t),\quad\forall t\in \left[ 0,\tau \right] . 
%\end{equation*}
%\end{corollary}

\subsection{Assumptions and main results}

In this subsection, we formulate the main assumptions that will be used in this note and state our main results as well. Let $A:D(A)\subset X\to X$ and $B:D(B)\subset X\to X$ be two closed and resolvent commuting linear operators. Here by resolvent commuting we mean that 
\begin{equation}\label{hyp-commute}
(\lambda-A)^{-1}(\mu-B)^{-1}=(\mu-B)^{-1}(\lambda-A)^{-1},\quad\forall \lambda,\mu\in \rho(A)\cap\rho(B),
\end{equation}
wherein $\rho(A)$ and $\rho(B)$ denote the resolvent set of $A$ and $B$, respectively.

As far as operator $A:D(A)\subset X\to X$ is concerned, we assume that it satisfies the following set of properties.
\begin{assumption}[Assumption on the linear operator $A$]\label{ASS1-A}
The closed linear operator $A$ satisfies Assumptions \ref{ASS1} with parameters $\omega_A\in\R$, $M_A\geq 1$ and Assumption \ref{ASS2} with some parameter $p\in [1,\infty)$ and, without loss of generality, $M=M_A$ and $\omega=\omega_A$, that reads as for all $f\in L^p_{\rm loc}\left([0,\infty); X\right)$ one has
\begin{equation*}
\left\|(S_A\diamond f)(t)\right\|\leq M_A \left(\int_0^t e^{p\omega_A(t-s)}\|f(s)\|^p \d s\right)^{\frac{1}{p}},\quad \forall t\geq 0.
\end{equation*} 
\end{assumption}

About the linear operator $B : D(B) \subset X \to X$, we assume that it satisfies the following set of assumptions.
\begin{assumption}[Assumption on the linear operator $B$]\label{ASS1-B}
The closed linear operator $B : D(B) \subset X \to X$ satisfies Assumption \ref{ASS-analytic} with parameters $\omega_B\in\R$ and $p^*\in [1,\infty)$.
\end{assumption}

In order to state our main result, we need an additional assumption coupling the properties of the linear operators $A$ and $B$, that reads as follows.
\begin{assumption}[Assumption coupling $A$ and $B$]\label{ASS1-C}
In addition to the two above assumptions, namely Assumption \ref{ASS1-A} and \ref{ASS1-B}, and the commutativity property \eqref{hyp-commute}, we assume that the parameters $p$ and $p^*$ satisfy
\begin{equation*}
\frac{1}{p}+\frac{1}{p^*}>1.
\end{equation*}
Setting  $q^*\in (1,\infty]$, the conjugate exponent of $p^*$ defined by $\frac{1}{p^*}+\frac{1}{q^*}=1$, the above inequality rewrites as
\begin{equation*}
\frac{1}{q^*}<\frac{1}{p}.
\end{equation*}
\end{assumption}

As recalled in the previous subsection, Assumption \ref{ASS1-A} allows us to introduce the strongly continuous semigroup $\{T_{A_0}(t)\}_{t\geq 0}\subset \mathcal L\left(\overline{D(A)}\right)$ generated by $A_0$, the part of $A$ in $\overline{D(A)}$, while Assumption \ref{ASS1-B} allows us to consider $\{S_B(t)\}_{t\geq 0}$ the integrated semigroup generated by $B$, $T_B(t) = \frac{\d S_B(t)}{\d t}$ as well as $\{T_{B_0}(t)\}_{t\geq 0}\subset \mathcal L\left(\overline{D(B)}\right)$, the analytic and strongly continuous semigroup generated by $B_0$, the part of $B$ in $\overline{D(B)}$. Let us observe that from the commutativity property \eqref{hyp-commute} one has, for all $t\geq 0$,
\begin{equation*}
T_{A_0}(t)\overline{D(A)\cap D(B)}\subset\overline{D(A)\cap D(B)},\quad T_{B_0}(t)\overline{D(A)\cap D(B)}\subset\overline{D(A)\cap D(B)},
\end{equation*}
and
\begin{equation*}
T_{A_0}(t)T_{B_0}(t)x=T_{B_0}(t)T_{A_0}(t)x, \quad\forall x\in \overline{D(A)\cap D(B)},\;\forall t\geq 0.
\end{equation*}
Using the above notations, the main result of this note reads as the follows. 
\begin{theorem}\label{THEO_sum}
Let Assumption \ref{ASS1-C} be satisfied. Consider the linear operator $A + B : D(A+B) \subset X \to X$ given by
\begin{equation*}
\left\lbrace\begin{aligned}
&D(A+B)=D(A)\cap D(B),\\
&(A+B)x=Ax+Bx,\quad\forall x\in D(A+B).
\end{aligned}\right.
\end{equation*}
Then, the following hold
\begin{itemize}
\item[{\rm (a)}] The linear operator $(A+B)$ is closable and there exists $\lambda_0\in\R$ such that for all $\lambda\geq \lambda_0$, we have
\begin{equation}\label{pq}
\overline{\overline{D(B)}+\overline{D(A)}}\subset \overline{{\rm R}\left(\lambda-\left(A+B\right)\right)}.
\end{equation}
Moreover, $A+B$ admits a closed extension, denoted by $\widehat{A+B}:D\left(\widehat{A+B}\right)\subset X\to X$ such that
$\rho\left(\widehat{A+B}\right)$ is not empty and contains some interval of the form $\left[\widehat{\lambda},\infty\right)$, for some $\widehat{\lambda}\in\R$.

\item[{\rm (b)}] If moreover we assume that there exists $\lambda_0\in\R$ such that, for all $\lambda\geq \lambda_0$, one has
\begin{equation}\label{cond}
{\rm R}\left(\lambda-\left(A+B\right)\right)\text{ is dense in }X,
\end{equation}
then the closure $\overline{A+B}:D\left(\overline{A+B}\right)\subset X\to X$ satisfies Assumption \ref{ASS1} and the following properties hold true:
\begin{itemize}
\item [{\rm (i)}] $\overline{D\left(\overline{A+B}\right)}=\overline{D(A)\cap D(B)}$.
\item[{\rm (ii)}] The strongly continuous semigroup generated by $\left(\overline{A+B}\right)_0$, the part of $\overline{A+B}$ in $\overline{D(A)\cap D(B)}$ and denoted by 
$$\left\{T_{\left(\overline{A+B}\right)_0}(t)\right\}_{t\geq 0}\subset \mathcal L\left(\overline{D(A)\cap D(B)}\right),$$ 
is given by
\begin{equation*}
T_{\left(\overline{A+B}\right)_0}(t)=T_{A_0}(t)_{|\overline{D(A)\cap D(B)}}\circ T_{B_0}(t)_{|\overline{D(A)\cap D(B)}},\quad\forall t\geq 0.
\end{equation*} 
\item[{\rm (iii)}] The integrated semigroup generated by $\overline{A+B}$ on $X$ is given by
$$
S_{\overline{A+B}}(t)=\left(S_A\diamond T_B(t-\cdot)\right)(t),\quad\forall t\geq 0.
$$
\item[{\rm (iv)}] The integrated semigroup $\left\{S_{\overline{A+B}}(t)\right\}$ satisfies the following regularity property: for any $r\in [1,\infty)$ with
\begin{equation*}
\frac{1}{r}<\frac{1}{p}+\frac{1}{p^*}-1,
\end{equation*}
 and any $\omega>\omega_0(B_0)$ there exists a continuous function $\delta:[0,\infty)\to [0,\infty)$ with $\delta(0)=0$ such that for any function $f\in L_{\rm loc}^{r}\left([0,\infty);X\right)$, the map $t\longmapsto \left(S_{\overline{A+B}}\ast f\right)(t)$ is continuously differentiable from $[0,\infty)$ into $X$ and the following estimate holds
\begin{equation*}
\|\left(S_{\overline{A+B}}\diamond f\right)(t)\|\leq \delta(t) \|f\|_{L^{r}(0,t;X)},\quad\forall t\geq 0.
\end{equation*}
\end{itemize}
\end{itemize}
\end{theorem}

\begin{remark}
Note that, because of \eqref{pq}, when $D(B)$ is dense in $X$ then \eqref{cond} is automatically satisfied and the above result has already been proved by Ducrot and Magal \cite{DM-TAMS} and Thieme \cite{Thieme08} in a more general framework. We also refer to Thieme \cite{Thieme97} for results in the same direction with Hille-Yosida operators.
Let us also notice that when $\overline{D(A)}+\overline{D(B)}$ is dense in $X$ then \eqref{cond} is satisfied. This point will be used latter to consider our application to age structured problem with diffusion, namely \eqref{eq-lin}.
\end{remark}

\begin{remark}
From the above theorem, one may observe (see ${\rm (ii)}$) that
\begin{equation*}
\omega_0\left(\left(\overline{A+B}\right)_0\right)\leq \omega_0 \left(A_{\overline{D(A)\cap D(B)}}\right) + \omega_0\left(B_{\overline{D(A)\cap D(B)}}\right)\leq \omega_0\left(A_0\right)+\omega_0\left(B_0\right),
\end{equation*}
wherein $A_{\overline{D(A)\cap D(B)}}$ and $B_{\overline{D(A)\cap D(B)}}$ denote respectively the part of $A$ and $B$ in $\overline{D(A)\cap D(B)}$. Hence we obtain that 
$$
\left(\omega_0\left(A_0\right)+\omega_0\left(B_0\right),\infty\right)\subset \left(\omega_0\left(A_{\overline{D(A)\cap D(B)}}\right)+\omega_0\left(B_{\overline{D(A)\cap D(B)}}\right),\infty\right),
$$
and, since $\rho\left(\overline{A+B}\right)=\rho\left(\left(\overline{A+B}\right)_0\right)$ (see Lemma 2.1 in \cite{Magal-Ruan09a}), one has
$$
\left(\omega_0\left(A_{\overline{D(A)\cap D(B)}}\right)+\omega_0\left(B_{\overline{D(A)\cap D(B)}}\right),\infty\right)\subset \rho\left(\overline{A+B}\right).
$$
As a special case, if $\omega_0\left(A_0\right)+\omega_0\left(B_0\right)<0$ then $0\in \rho\left(\overline{A+B}\right)$ and for any $y\in X$, the equation $-\left(A+B\right)x=y$ has a unique weak solution $x\in \overline{D(A)\cap D(B)}$, in the sense of Definition \ref{DEF-sol} below.  
\end{remark}

\begin{remark}\label{REM-res}
From our proof given below (see Lemma \ref{LE-estiR}), we shall obtain the following estimate for the resolvent of $\overline{A+B}$ 
\begin{equation*}
\left\|\left(\lambda-\overline{A+B}\right)^{-1}\right\|_{\mathcal L(X)}=O\left(\frac{1}{\lambda^{\gamma}}\right)\quad\text{as}\quad \lambda\to \infty,
\end{equation*}
for any $0<\gamma<\frac{1}{p}-\frac{1}{q^*}$.
\end{remark}

As a special case of Theorem \ref{THEO_sum} we obtain the following corollary.
\begin{corollary}%\label{COR1}
Let Assumption \ref{ASS1-C} be satisfied. Assume furthermore that
\begin{itemize}
\item[{\rm (i)}] There exists $\lambda_0\in\R$ such that for all $\lambda\geq \lambda_0$, one has
\begin{equation*}
{\rm R}\left(\lambda-\left(A+B\right)\right)\text{ is dense in }X.
\end{equation*}
\item[{\rm (ii)}] The strongly continuous semigroup 
$$\left\{T_{A_0}(t)_{|\overline{D(A)\cap D(B)}}\circ T_{B_0}(t)_{|\overline{D(A)\cap D(B)}}\right\}_{t\geq 0}\subset \mathcal L\left( \overline{D(A)\cap D(B)}\right)$$ 
is analytic.
\end{itemize}
Then, $A+B$ is closable and its closure, denoted by $\overline{A+B}:D\left(\overline{A+B}\right)\subset X\to X$, satisfies Assumption \ref{ASS1-B} with $\frac{1}{p^*}$ replaced by $\gamma\in \left(0,\frac{1}{p}-\frac{1}{q^*}\right)$. In particular the linear operator $\overline{A+B}$ is $\gamma$-almost sectorial for any $\gamma\in \left(0,\frac{1}{p}-\frac{1}{q^*}\right)$.
\end{corollary} 

The above corollary directly follows from Remark \ref{REM-res} coupled with Proposition 3.3 in \cite{Ducrot-Magal-Prevost09}.

A more specific situation is concerned with the commutative sum of two operators satisfying both Assumption \ref{ASS1-B}. This result is presented in the next corollary.
\begin{corollary}%\label{COR2}
Assume that both linear operators $A$ and $B$ satisfy Assumption \ref{ASS1-B} respectively with exponent $p_A^\ast\in [1,\infty)$ and $p_B^\ast \in [1,\infty)$.
Then, if there exists $\lambda_0\in\R$ such that, for all $\lambda\geq \lambda_0$, one has
\begin{equation*}
{\rm R}\left(\lambda-\left(A+B\right)\right)\text{ dense in }X,
\end{equation*}
and
\begin{equation*}
\frac{1}{p_A^\ast}+\frac{1}{p_B^\ast}>1,
\end{equation*}
then, $A+B:D(A)\cap D(B)\subset X\to X$ is closable and its closure, denoted by $\overline{A+B}:D\left(\overline{A+B}\right)\subset X\to X$, also satisfies Assumption \ref{ASS1-B} with any exponent $p^*\in (1,\infty)$ such that 
$$
0<\frac{1}{p^\ast}< \frac{1}{p_A^\ast}+\frac{1}{p_B^\ast}-1\leq 1.
$$
\end{corollary}

\section{Proof of Theorem \ref{THEO_sum}}

This section is devoted to the proof of Theorem \ref{THEO_sum}. To that aim we closely follow the construction provided by Ducrot and Magal in \cite{DM-TAMS} to handle the closability of the commutative sum of two linear operators, one satisfying Assumption \ref{ASS1-A} while the other is the generator of a strongly continuous semigroup.
Here we extend this methodology to take care of the additional difficulty coming from the fact that the linear operator $B$ is, as $A$, also not densely-defined and that the derivative semigroup $T_B(t)=\frac{\d S_B}{\d t}(t)$ admits a singularity at $t=0$. Here we shall crucially make use the fact that such a singularity is of polynomial type and thus belongs to a suitable Lebesgue space.

To perform our analysis and prove Theorem \ref{THEO_sum}, we shall make use of exponentially weighted Lebesgue and Sobolev spaces. More specifically, for any $p\in [1,\infty)$ and $\eta\in \R$ we define the Banach space $L^p_\eta(0,\infty;X)$ by
\begin{equation*}
L^p_\eta(0,\infty;X)=\left\{\psi\in L^1_{\rm loc}([0,\infty);X):\; t\mapsto e^{\eta t}\psi(t)\in L^p(0,\infty;X)\right\},
\end{equation*}
endowed with the weighted norm, denoted by $\|\cdot\|_{L^q_\eta(0,\infty;X)}$, and given by
\begin{equation*}
\|\psi\|_{L^p_\eta(0,\infty;X)}=\| e^{\eta \cdot}\psi\|_{L^p(0,\infty;X)},\quad\forall \psi\in L^p_\eta(0,\infty;X).
\end{equation*}
We also introduce the associated weighted Sobolev spaces $W^{1,p}_\eta(0,\infty;X)$ given by
\begin{equation*}
W^{1,p}_\eta(0,\infty;X)=\left\{\psi\in W^{1,p}_{\rm loc}([0,\infty);X):\; \left(\psi,\psi'\right)\in L^p_\eta\hspace{-0.05cm}\left(0,\infty;X\right)^2\hspace{-0.05cm}\right\}\hspace{-0.05cm}.
\end{equation*}
Before proving Theorem \ref{THEO_sum} we need to prove several preliminary lemmas and results. 
Our first lemma reads as follows.
\begin{lemma}\label{LE4.1}
Let Assumption \ref{ASS1-A} be satisfied. Recalling that $\omega_A$, $M_A$ and $p$ are defined in Assumption \ref{ASS1-A}, let $\eta>\max\left(0,\omega_A\right)$ be given. Then, for each $f\in L_{\eta}^p(0,\infty; X)$ the following limit exists in $\overline{D(A)}$
$$
\lim_{t\to\infty}\left(S_A\diamond f(t-\cdot)\right)(t).
$$
Moreover the operator $\mathbb{K}_A:L_{\eta}^p\left(0,\infty; X \right)\to \overline{D(A)}$ defined by
$$
\mathbb{K}_A(f)=\lim_{t\to\infty}\left(S_A\diamond f(t-\cdot)\right)(t),
$$
is a bounded linear operator and satisfies the following estimate
\begin{equation*}
\left\|\mathbb{K}_A\right\|_{\mathcal L\left(L_{\eta}^p\left(0,\infty; X \right),\overline{D(A)}\right)}\leq \dfrac{M_A^{1+\frac{1}{p}}}{1-e^{\left(\omega_A-\eta \right)}} =: \widetilde{M}_A\left(\eta \right).
\end{equation*}
\end{lemma}
 
\begin{proof}
Fix $\eta >\max(0,\omega_A)$ and $f\in L_{\eta}^p(0,\infty; X)$. Next in order to prove the above result, let us show that the function $t\mapsto \left(S_A\diamond f(t-\cdot)\right)(t)$ satisfies the Cauchy criterion as $t\to\infty$.
To that aim, let us first observe that due to the constant variation formula \eqref{constant-varf}, for all $0\leq s\leq t$, one has
\begin{equation*}
(S_A\diamond f)(t)=T_{A_0}(t-s)(S_A\diamond f)(s)+(S_A\diamond f(s+\cdot))(t-s).
\end{equation*}
Hence one gets, for all $t>0$ and $h>0$,
\begin{equation*}
\begin{split}
(S_A\diamond f(t+h-.))(t+h) \;=\; &T_{A_0}(t+h-h)(S_A\diamond f(t+h-\cdot))(h) \\ 
&+ (S_A\diamond f(t+h-h-.))(t+h-h),
\end{split}
\end{equation*}
so that, for any $t>0$ and $h>0$, one has
\begin{equation}\label{esti1}
\begin{split}
(S_A\diamond f(t+h-.))(t+h)-(S_A\diamond f(t&-\cdot))(t)\\
&= T_{A_0}(t)(S_A\diamond f(t+h-\cdot))(h).
\end{split}
\end{equation}
Now, for any $t>0$ and $h>0$, one has
\begin{equation*}
\|T_{A_0}(t)(S_A\diamond f(t+h-\cdot))(h)\|\leq \|T_{A_0}(t)\|_{\mathcal L(\overline{D(A)})}\|(S_A\diamond f(t+h-\cdot))(h)\|,
\end{equation*}
Next, since the strongly continuous semigroup $\{T_{A_0}(t)\}_{t\geq 0}$ satisfies
\begin{equation*}
\|T_{A_0}(t)\|_{\mathcal L(\overline{D(A)})}\leq M_A e^{\omega_A t},\quad\forall t\geq 0,
\end{equation*}
the estimate in Assumption \ref{ASS1-A} leads us to
\begin{equation*}
\begin{split}
&\|T_{A_0}(t)(S_A\diamond f(t+h-\cdot))(h)\|\\
& \leq M_{A} e^{\omega_{A}t}\left(M_{A} \int_0^h e^{p\omega_A(h-l)} \|f(t+h-l)\|^p \d l\right)^{1/p} \\
& \leq M_{A}^{1+\frac{1}{p}} e^{\omega_{A}t} \left( \int_0^h e^{p\omega_A(h-l)} e^{-p\eta(t+h-l)} e^{p\eta(t+h-l)}\|f(t+h-l)\|^p \d l\right)^{1/p},
\end{split}
\end{equation*}
hence
\begin{equation*}
\begin{split}
&\|T_{A_0}(t)(S_A\diamond f(t+h-\cdot))(h)\|\\
&\leq M_{A}^{1+\frac{1}{p}} e^{\left(\omega_{A}-\eta\right)t} \left( \int_0^h e^{p(\omega_A - \eta)(h-l)}  e^{p\eta(t+h-l)}\|f(t+h-l)\|^p \d l\right)^{1/p}.
\end{split}
\end{equation*}
Recalling that $\eta>\max(0,\omega_A)$, we infer from \eqref{esti1} and the above inequality that, for all $t>0$ and $h>0$,
\begin{equation*}
\begin{split}
&\left\|(S_A \diamond f(t+h-\cdot))(t+h)-(S_A\diamond f(t-\cdot))(t)\right\| = \|T_{A_0}(t)(S_A\diamond f(t+h-\cdot))(h)\| \\
&\leq M_{A}^{1+\frac{1}{p}} e^{(\omega_{A}-\eta) t} \left( \int_0^h  e^{p\eta(t+h-l)}\|f(t+h-l)\|^p \d l\right)^{\frac{1}{p}} \\
&\leq M_{A}^{1+\frac{1}{p}}e^{(\omega_{A}-\eta)t} \left(\int_t^{t+h} e^{p\eta \sigma}\|f(\sigma)\|^p \d\sigma\right)^{\frac{1}{p}} \leq M_{A}^{1+\frac{1}{p}} e^{(\omega_{A}-\eta)t} \left\|f\right\|_{L^p_\eta(0,\infty;X)}.
\end{split}
\end{equation*}
Since $\omega_A-\eta<0$, the above estimate ensures that the expect limit do exist.

Now we derive the estimate for $\mathbb{K}_A$.
To do, note that since $\omega_A-\eta<0$ and $(S_A \diamond f(0-\cdot))(0)=0$,  it follows that for each $k\in\mathbb N$ and $t \in [k,k+1]$ one has
$$
\begin{array}{ll}
\left\|(S_A \diamond f(t-\cdot))(t)\right\|& \leq  \left\|(S_A \diamond f(t-\cdot))(t)-(S_A\diamond f(k-\cdot))(k)\right\| \\
&+ \left\|(S_A \diamond f(k-\cdot))(k)-(S_A\diamond f((k-1)-\cdot))(k-1)\right\|\\
&\vdots  \\
&+\left\|(S_A \diamond f(1-\cdot))(1)-(S_A \diamond f(0-\cdot))(0)\right\|  \\
&\leq M_{A}^{1+\frac{1}{p}} \left\|f\right\|_{L^p_\eta(0,\infty;X)} \left[e^{(\omega_{A}-\eta)k}+e^{(\omega_{A}-\eta)(k-1)}+\ldots+1 \right]\\
&\leq \frac{M_{A}^{1+\frac{1}{p}}}{1-e^{\omega_A-\eta}} \left\|f\right\|_{L^p_\eta(0,\infty;X)},
\end{array}
$$
and the result follows. 
\end{proof}

By using the approximation formula (see Magal and Ruan \cite{Magal-Ruan07})
$$
\frac{\d }{\d t}(S_A\ast f)(t)= \lim_{\lambda \to \infty} \int_0^t T_{A_0}(t-s) \lambda(\lambda-A)^{-1}f(s)\d s,\quad\forall t\geq 0,
$$
we deduce that for each $\tau>0$, $f\in L^p(0,\tau;X)$ and $\delta\in\R$ \begin{equation}\label{KA-delta (f) = KA(e-delta f)}
\left(S_{A-\delta}\diamond f\left(t-. \right)\right)(t)=\left(S_A\diamond \left(e^{-\delta\cdot}f\right) \left(t-. \right)\right)(t),\quad\forall t\in [0,\tau].
\end{equation}
The above formula coupled with Lemma \ref{LE4.1} yields the following lemma.
\begin{lemma}\label{LE3.5}
Let $\delta\in\R$ and $\eta>\max(0,\omega_A)$ be given. Then, for all function $f\in L^p_{\eta-\delta}(0,\infty;X)$, the limit 
$$
\mathbb{K}_{A-\delta}(f):=\lim_{t\to\infty}\left(S_{A-\delta}\diamond f \left(t-. \right)\right)(t)\text{ exists in }\overline{D(A)}, 
$$
and for all $f\in L^p_{\eta-\delta}(0,\infty;X)$ one has
$$
\mathbb{K}_{A-\delta}(f)=\mathbb{K}_A\left(e^{-\delta\cdot}f \right),
$$ 
and
\begin{equation*}
\left\|\mathbb{K}_{A-\delta}(f)\right\|\leq \widetilde{M}_A\left(\eta \right) \|f\|_{L_{\eta-\delta}^p\left(0,\infty; X \right)}.
\end{equation*}
\end{lemma}

\begin{proof}
The proof of this lemma directly follows from the application of Lemma \ref{LE4.1} together with \eqref{KA-delta (f) = KA(e-delta f)} by noticing that
\begin{equation*}
f\in L^p_{\eta-\delta}(0,\infty;X)\;\Longleftrightarrow\;e^{-\delta\cdot}f\in L^p_\eta(0,\infty;X),
\end{equation*}
\end{proof}

Our next lemma reads as follows.
\begin{lemma}\label{LE4.2}
Let $\eta>\max\left(0,\omega_A\right)$ be given. Then, for each $f\in W^{1,p}_\eta(0,\infty;X)$, one has
\begin{equation*}
\mathbb{K}_A(f)\in D(A).
\end{equation*}
Moreover, the following identity holds true 
\begin{equation*}
\mathbb{K}_{A}(f^{\prime })+ A\mathbb{K}_{A}(f) + f(0) = 0,
\end{equation*}
\end{lemma}
\begin{proof}
Applying Assumption \ref{ASS1-A} with any constant function $f(t)\equiv x$ for $x\in X$ and noticing that in such a case one has: 
\begin{equation*}
\left( S_{A}\diamond f\right) (t)=S_{A}(t)x,
\end{equation*}
one obtains that $S_A(t)$ satisfies
\begin{equation}\label{pm}
\left\Vert S_{A}(t)x\right\Vert \leq M_A\left(\int_0^t e^{p\omega_As}\d s\right)^{1/p}\|x\|,\quad\forall t\geq 0,\;\forall x\in X.
\end{equation}%
Now let $f\in W^{1,p}_{\eta }\left( \left[ 0,\infty \right) ,X\right)$ be given. Then, for all $t\geq 0$, one has 
\begin{equation*}
\|f(t)\|\leq\int_t^\infty \|f'(s)\|\d s\leq \int_t^\infty e^{-\eta s} \left(e^{\eta s}\|f'(s)\|\right)\d s.
\end{equation*}
Next set $q\in (1,\infty]$, the conjugate exponent of $p$, namely $\frac{1}{p}+\frac{1}{q}=1$ and using H\"older inequality yields, for any $t\geq 0$,
\begin{equation*}
\|f(t)\|\leq \left[\int_t^\infty e^{-\eta q s}\d s\right]^{1/q}\| e^{\eta\cdot}f'\|_{L^p(0,\infty;X)} \leq C(\eta,q)e^{-\eta t}\| e^{\eta\cdot}f'\|_{L^p(0,\infty;X)}.
\end{equation*}
Herein the constant $C(\eta,q)$ is given by
$C(\eta,q)=(q\eta)^{-1/q}$.
Next, applying \eqref{pm} with $x=f(t)$, one gets
\begin{equation*}
\|S_A(t)f(t)\|\leq M_A\left[\int_0^t e^{p\left(\omega_A s-\eta t\right)}\d s\right]^{1/p}C(\eta, q)\| e^{\eta\cdot}f'\|_{L^p(0,\infty;X)},\quad\forall t\geq 0.
\end{equation*}
We infer from the condition $\eta> \max(0,\omega_A)$ that 
\begin{equation}\label{rt}
S_A(t)f(t)\to 0 \quad\text{as}\quad t\to\infty.
\end{equation}
On the one hand, let us observe that, since $f\in W^{1,p}_{\eta }\left( \left[0,\infty \right) ,X\right)$, for all $t > 0$, one has 
\begin{equation*}
\left( S_{A}\diamond f\right) (t)=S_{A}(t)f(0)+\left( S_{A}\ast f^{\prime
}\right) (t)=S_{A}(t)f(0)+\int_{0}^{t}\left( S_{A}\diamond f^{\prime
}\right) (s)ds.
\end{equation*}
Thus, this yields 
\begin{eqnarray*}
\left( S_{A}\diamond f(t-.)\right) (t) &=&S_{A}(t)f(t)-\left( S_{A}\ast f^{\prime }(t-.)\right) (t) \\
&=&S_{A}(t)f(t)-\int_{0}^{t}S_{A}(t-s)f^{\prime }(t-s)ds \\
&=&S_{A}(t)f(t)-\int_{0}^{t}\left( S_{A}\diamond f^{\prime }(t-.)\right)(s) ds.
\end{eqnarray*}
The above computation coupled together with \eqref{rt} and the definition of $\mathbb K_A(f)$ in Lemma \ref{LE4.1} yields
\begin{equation*}
\mathbb{K}_{A}(f)=-\lim_{t\rightarrow \infty }\int_{0}^{t}\left(
S_{A}\diamond f^{\prime }(t-.)\right) (s)ds.
\end{equation*}%
On the other hand, note that the map $t\rightarrow \left( S_{A}\diamond f^{\prime
}(t-.)\right) (t)$ satisfies, for any $t\geq 0$, the following equation
\begin{eqnarray*}
\left( S_{A}\diamond f^{\prime }(t-.)\right) (t) &=&A\int_{0}^{t}\left(S_{A}\diamond f^{\prime }(t-.)\right) (s)\d s+\int_{0}^{t}f^{\prime }(t-s)\d s \\
&=&A\int_{0}^{t}\left( S_{A}\diamond f^{\prime }(t-.)\right) (s)\d s+f(t)-f(0).
\end{eqnarray*}
Finally, since $A$ is closed and $f'\in L^p_\eta(0,\infty;X)$, letting $t\to\infty$ and using Lemma~\ref{LE4.1} (with $f$ replaced by $f'$) yields $\mathbb{K}_A(f)\in D(A)$ and the following identity
\begin{equation*}
\mathbb{K}_{A}(f')=-A\mathbb{K}_{A}(f) - f(0).
\end{equation*}
This completes the proof of the result.
\end{proof}

\begin{lemma}\label{LE4.3}
Fix $\eta >\max(0,\omega_A)$. Then, for each $f\in L^p_\eta\left(0,\infty;D(B)\right)$ one has
\begin{equation*}
\mathbb{K}_A(f)\in D(B) \quad \text{and} \quad \mathbb{K}_A(Bf) = B \mathbb{K}_A(f).
\end{equation*}
Here the notation $f\in L^p_\eta\left(0,\infty;D(B)\right)$ is used to mean $f\in L^p_\eta\left(0,\infty;X\right)$, $f(t)\in D(B)$ for almost any $t\geq 0$ and $t \to Bf(t)$ belongs to $L^p_\eta\left(0,\infty;X\right)$.\end{lemma}

The proof of this lemma is similar to the one of Lemma 4.2 in \cite{DM-TAMS} and thus omitted.

To go further, we introduce more notations related to the linear operator $B$.
According to Assumption \ref{ASS1-B} we denote by $\left\{S_{B}(t)\right\}_{t\geq 0}\subset \mathcal L(X)$ the integrated semigroup generated by $B$. We also consider $B_0$, the part of $B$ in $\overline{D(B)}$, $\{T_{B_0}(t)\}_{t\geq 0}\subset \mathcal L\left(\overline{D(B)}\right)$ the strongly continuous semigroup generated by $B_0$ as well as $\{T_B(t)\}_{t>0}\subset \mathcal L(X)$ the semigroup obtained by differentiating $S_B(t)$ with respect to time, for $t>0$. 

Recall also that for each $\omega>\omega_0(B_0)$ and each $\alpha\in \left(\frac{1}{q^*},1\right)$, there exists some constant $M=M_{\alpha,\omega}>0$ such that
\begin{equation}\label{esti-TB}
\|T_{B}(t)\|_{\mathcal L(X)}\leq \frac{M}{t^\alpha}e^{\omega t},\quad\forall t>0.
\end{equation} 
Now fix $\eta>\max(0,\omega_A)$ and recalling Assumption \ref{ASS1-C}, let us fix $\alpha$ such that
\begin{equation*}
\frac{1}{q^*}<\alpha<\frac{1}{p}.
\end{equation*}
With such a choice of the parameter $\alpha$, observe that for any $x\in X$ and any $\lambda>\eta+\omega_0(B_0)$ the map $t\mapsto e^{-\lambda t}T_{B}(t)x\in L^p_\eta\left(0,\infty;X\right)$.
Hence, in view of Lemma \ref{LE4.1}, let us define for $\lambda>\eta+\omega_0(B_0)$ the map $\mathcal R\left(\lambda,A,B\right)\in \mathcal L(X,\overline{D(A)})$ by
\begin{equation*}
\begin{split}
\mathcal R\left(\lambda,A,B\right) x=&\lim_{t\to \infty}\left[S_A\diamond \left(e^{-\lambda (t-\cdot)}T_{B}(t-\cdot)x\right)\right](t), \quad \forall x\in X\\
&=\mathbb{K}_A\left(e^{-\lambda\cdot}T_{B}(\cdot)x\right).
\end{split}
\end{equation*}
Let us observe that, using Lemma \ref{LE3.5}, the following identities hold 
\begin{equation*}
\mathcal R\left(\lambda,A,B\right)=\mathcal R\left(0,A,B-\lambda\right)=\mathcal R\left(0,A-\lambda,B\right),\quad \forall \lambda>\eta+\omega_0(B_0).
\end{equation*}

Next the following estimate holds true
\begin{lemma}\label{LE-estiR}
For any $\alpha\in\left(\frac{1}{q^*},\frac{1}{p}\right)$, for each $\omega>\omega_0(B_0)$, there exists some constant $M>0$ such that, for any $\lambda>\eta+\omega$, one has
\begin{equation}\label{esti-R}
\left\|\mathcal R\left(\lambda,A,B\right)\right\|_{\mathcal L(X,\overline{D(A)})}\leq \frac{M}{\left(\lambda-(\eta+\omega)\right)^{\frac{1}{p}-\alpha}}.
\end{equation}
In particular one has 
\begin{equation*}
\lim_{\lambda\to \infty} \left\|\mathcal R\left(\lambda,A,B\right)\right\|_{\mathcal L(X,\overline{D(A)})}=0.
\end{equation*}
\end{lemma}

\begin{proof}
Fix $\alpha\in\left(\frac{1}{q^*},\frac{1}{p}\right)$ and $\omega>\omega_0(B_0)$ and recall that there exists $M=M_{\alpha,\omega}$ such that \eqref{esti-TB} holds true.
Next using Lemma \ref{LE4.1}, one has for any $x\in X$ and $\lambda>\eta+\omega$,
\begin{equation*}
\left\|\mathcal R(\lambda,A,B)x\right\|\leq \widetilde{M}_A\left(\eta \right) \left\|e^{-\lambda\cdot}T_B(\cdot)x\right\|_{L^p_\eta(0,\infty;X)}.
\end{equation*}
We infer from \eqref{esti-TB} that
\begin{equation*}
\left\|\mathcal R(\lambda,A,B)x\right\|\leq \widetilde{M}_A\left(\eta \right) M\left[\int_0^\infty e^{-p\left(\lambda-(\eta+\omega)\right) t}\frac{1}{t^{\alpha p}} \d t \right]^{\frac{1}{p}}\|x\|.
\end{equation*}
The change of variable $s=\left(\lambda-(\eta+\omega)\right)t$ ensures that estimate \eqref{esti-R} holds true and this completes the proof of the lemma.
\end{proof}

Our next preliminary result reads as follows. 
\begin{lemma}\label{LE-REG}
For all $\lambda>\eta+\omega_0\left(B_0\right)$ one has
\begin{equation*}
\mathcal R(\lambda,A,B) D(B)\subset D(A)\cap D(B),
\end{equation*}
and 
\begin{equation*}
\left(\lambda -(A+B)\right)R\left(\lambda,A,B\right) x=x,\quad\forall x\in D(B).
\end{equation*}
\end{lemma}

\begin{proof}
To prove the above lemma, let us fix $\lambda>\eta+\omega_0(B_0)$.
Next fix $x\in D(B)$ and let us check that the map 
$$t\mapsto e^{-\lambda t}T_{B}(t)x = e^{-\lambda t}T_{B_0}(t)x$$ 
belongs both to $W^{1,p}_\eta(0,\infty;X)$ and to $L^p_\eta\left(0,\infty;D(B)\right)$.
Note that the first statement follows from \eqref{esti-2nd}. The second statement directly follows from the first one coupled with the following property of the semigroup $T_{B_0}(t)$, that reads as
$$
T_{B_0}(t)x\in D(B_0),\quad\forall t>0\quad\text{and}\quad\frac{\d T_{B_0}(t)x}{\d t}=B_0 T_{B}(t)x,\quad\forall t>0.
$$
Hence (see \eqref{esti-2nd}) $t\mapsto e^{-\lambda t}T_{B_0}(t)x$ belongs to $L^p_\eta(0,\infty;D(B_0))\subset L^p_\eta(0,\infty;D(B))$.
As a consequence of Lemma \ref{LE4.2} and \ref{LE4.3}, one obtains that
$$
\mathcal R\left(\lambda,A,B\right) D(B)\subset D(A)\cap D(B),\quad\forall \lambda>\eta+\omega_0(B_0).
$$
The second statement of the result follows from the same arguments as the ones developed in \cite{DM-TAMS}.
\end{proof}

We continue this section by proving the following proposition.

\begin{proposition}
Let Assumption \ref{ASS1-C} be satisfied. Then there exists $\lambda_0 \in \mathbb{R}$ such that, for all $\lambda \geq \lambda_0$, one has
$$\overline{\overline{D(B)} + \overline{D(A)}} \subset \overline{{\rm R}(\lambda-(A+B))}.$$
\end{proposition}

\begin{proof}
To prove the above proposition, let us define the part of $B$ in $\overline{D(A)}$, that is linear operator $B_1 : D(B_1) \subset \overline{D(A)} \rightarrow \overline{D(A)}$ given by
$$\left\lbrace\begin{array}{l}
D(B_1) = \{ x \in D(B) \cap \overline{D(A)} : B x \in \overline{D(A)} \}, \\
B_1 x = B x, \quad \forall ~x \in D(B_1).
\end{array}\right.$$
Let us also recall that the linear operator $A_0 : D(A_0) \subset \overline{D(A)} \rightarrow \overline{D(A)}$, defined in \eqref{Def A0}, is the part of $A$ in $\overline{D(A)}$. Then, from Assumption \ref{ASS1}, $A_0$ is the infinitesimal generator of a strongly continuous semigroup $\{T_{A_0}(t)\}_{t\geq 0}$.

Moreover, thanks to the commutativity property \eqref{hyp-commute}, for all $\lambda \in \rho(B)$, we have
$$(\lambda - B)^{-1} \overline{D(A)} \subset \overline{D(A)}.$$
Then, we deduce that $B_1=B_{|\overline{D(A)}}$ and it follows that
$$
\rho(B_1) = \rho(B) \quad \text{and} \quad (\lambda - B_1)^{-1} = (\lambda - B)^{-1}_{|\overline{D(A)}}, \quad \forall \lambda\in \rho(B_1) = \rho(B).$$
Due to Assumption \ref{ASS-analytic}, the linear operator $B_1$ is the infinitesimal generator of a strongly continuous analytic semigroup of bounded linear operators on $\overline{D(B_1)} = \overline{D(B) \cap \overline{D(A)}} \subset \overline{D(B)}$ and the following estimate holds true$$\|(\lambda - B_1)^{-1}\|_{\mathcal{L}(\overline{D(A)})} \leq \|(\lambda - B)^{-1}\|_{\mathcal{L}(X)} = \mathcal O\left(\frac{1}{\lambda^{\frac{1}{p*}}}\right)\text{ as }\lambda\to\infty.$$
This implies that the linear operator $B_1$ is the infinitesimal generator of an integrated semigroup $\left(S_{B_1}(t)\right)_{t\geq 0}$ that is given by $S_{B_1}(t) = S_B (t)_{|\overline{D(A)}}$, for all $t\geq 0$.

Moreover, from Assumption \ref{ASS2} and Remark \ref{REM-B}, $\{S_1(t)\}_{t\geq 0}$ satisfies Assumption \ref{ASS1-A} for any $p>p^*$ on the Banach space $\overline{D(A)}$.

Next using Theorem 4.12 in \cite{DM-TAMS}, the commutative sum $A_0 + B_1$ is closable and there exists $\tilde{\omega} \in \mathbb{R}$, such that $[\tilde{\omega},\infty) \subset \rho(\overline{A_0+B_1})$ wherein $\overline{A_0+B_1}$ denotes the closure of $A_0+B_1:D(A_0)\cap D(B_1)\subset \overline{D(A)}\to \overline{D(A)}$. Thus, for any $\lambda > \tilde{\omega}$ one obtains that
$${\rm R}\left(\lambda-\left(\overline{A_0+B_1}\right)\right)=\overline{D(A)},$$
which implies
\begin{equation}\label{DA barre inclusion}
\overline{D(A)}=\overline{{\rm R}\left(\lambda - (A_0 + B_1)\right)} \subset \overline{{\rm R}\left(\lambda - (A + B)\right)},\quad \forall \lambda\geq \tilde\lambda.
\end{equation}
Furthermore Lemma \ref{LE-REG} ensures that
\begin{equation}\label{DB barre inclusion}
D(B)\subset {\rm R}\left(\lambda - (A + B)\right),\quad \forall \lambda>\eta+\omega_0(B_0). 
\end{equation}
Finally, \eqref{DA barre inclusion} and \eqref{DB barre inclusion} completes the proof of the proposition.
\end{proof}

\quad

In the sequel, we fix $\lambda_0>\eta+\omega_0(B_0)$ and we study the solvability of the equation for $y\in X$ and $\widehat \lambda\geq \lambda_0$
\begin{equation}\label{EQ}
-\left(A+B-\widehat \lambda\right)x=y.
\end{equation}
Following \cite{DM-TAMS} we recall the notion of solution for \eqref{EQ}.\begin{definition}[Weak and Classical solutions]\label{DEF-sol}
Let $y\in X$ and $\widehat \lambda\geq \lambda_0$ be given. We will say that $x\in X$ is\textbf{\
a weak solution} of \eqref{EQ} if $x$ and $y$ satisfy the following equality 
%\begin{equation}\label{4.4}
$$
\left[ \left( \mu+\widehat\lambda -B\right) ^{-1}+\left( \lambda -A\right) ^{-1}\right]
x=\left( \mu+\widehat \lambda -B\right) ^{-1}\left( \lambda -A\right) ^{-1}\left[ y+\left(
\lambda +\mu \right) x\right]  
$$
%\end{equation}
for some $\lambda \in \rho \left( A\right) $ and $\mu \in \mathbb{R}$ such that $\mu + \widehat\lambda \in \rho(B).$

We will say that $x$ is \textbf{a classical solution} of \eqref{EQ} if 
\begin{equation*}
x\in D(A)\cap D(B)\quad \text{and} \quad-\left[ A+B-\widehat\lambda\right] x=y.
\end{equation*}
\end{definition}

As proved by Ducrot and Magal in \cite{DM-TAMS} and using Lemma \ref{LE-REG}, one has
\begin{lemma}\label{LEGraph}
The following holds true:
\begin{itemize}
\item[{\rm (i)}] Let $y\in X$ and $\widehat\lambda\geq \lambda_0$ be given. Then, the following properties hold
\begin{itemize}
\item[(i1)] If $x\in X$ is a classical solution of \eqref{EQ}, then $x$ is a weak solution according to Definition \ref{DEF-sol}.
\item[(i2)] $x\in X$ is a weak solution of \eqref{EQ} if and only if $x=\mathcal R\left(\widehat\lambda,A,B\right)y$.
\end{itemize}
\item[{\rm (ii)}] Consider the closed set $G\subset X\times X$ given by
\begin{equation*}
G=\left\{(x,y)\in X\times X:\;x=\mathcal R\left(\widehat \lambda,A,B\right)y\right\}.
\end{equation*}
Then, there exists a closed linear operator $L_{\widehat\lambda}:D\left(L_{\widehat\lambda}\right)\subset X\to X$ such that $G={\rm Graph}\,\left(-L_{\widehat \lambda}\right)$, $0\in\rho\left(L_{\widehat \lambda}\right)$ and $\left(-L_{\widehat \lambda}\right)^{-1}=\mathcal R\left(\widehat\lambda,A,B\right)$.
\end{itemize}
\end{lemma}

We are now able to prove the closability of $A+B$. We will more particularly prove the following result.
\begin{proposition}%\label{PROP-closability}
Let Assumption \ref{ASS1-C} be satisfied. Then the linear operator $A+B:D(A)\cap D(B)\subset X\to X$ is closable and its closure, denoted by $\overline{A+B}:D\left(\overline{A+B}\right)\subset X\to X$ satisfies for all $\widehat\lambda\geq \lambda_0$
\begin{equation*}
\left\lbrace\begin{aligned}
&D(A)\cap D(B)\subset D\left(\overline{A+B}\right)\subset \overline{D(A)\cap D(B)},\quad D\left(\overline{A+B}\right)\subset D\left(L_{\widehat \lambda}\right),\\
&\overline{A+B}\,x=\left(L_{\widehat{\lambda}}+\widehat\lambda\right)x,\quad\forall x\in D\left(\overline{A+B}\right).
\end{aligned}\right.
\end{equation*}
Let us furthermore assume that there exists $\widehat\lambda\geq \lambda_0$ such that
$${\rm R}\left(\widehat\lambda-(A+B)\right)\quad \text{is dense in }X.$$
Then, the closure $\overline{A+B}$ coincides with $L_{\widehat\lambda}+\widehat\lambda$, namely $D\left(L_{\widehat\lambda}\right)=D\left(\overline{A+B}\right)$. Moreover one also has $\widehat\lambda\in \rho\left(\overline{A+B}\right)$ and
\begin{equation*}
\left(\widehat\lambda-\left(\overline{A+B}\right)\right)^{-1}=\mathcal R(\widehat\lambda,A,B).
\end{equation*}
\end{proposition}

\begin{proof}
We first prove the first part of the proposition, namely the closability part.
To that aim, fix $\widehat\lambda\geq \lambda_0$
and consider $G_0={\rm Graph}\,\left(A+B-\widehat\lambda\right)$. Then, from $(i1)$ in Lemma \ref{LEGraph} and recalling that $L_{\widehat\lambda}$ is a closed linear operator one gets
\begin{equation*}
G_0\subset G\;\Longrightarrow\;\overline{G_0}\subset {\rm Graph}\,\left(L_{\widehat\lambda}\right). 
\end{equation*}
This means that $L_{\widehat\lambda}:D\left(L_{\widehat{\lambda}}\right)\subset X\to X$ is an invertible closed extension of $(A+B-\widehat\lambda)$ and this completes the proof of the first part of the result.

To prove the second part, assume now there exists $\widehat\lambda\geq\lambda_0$ such that the range $\left(\widehat\lambda-(A+B)\right)\left[D(A)\cap D(B)\right]$ is dense in $X$. 

Fix $(x,y)\in {\rm Graph}\,\left(L_{\widehat\lambda}\right)$, that is $x\in D\left(L_{\widehat{\lambda}}\right)$ and $y=L_{\widehat{\lambda}}x$, namely $x=-\mathcal R\left(\widehat \lambda, A,B\right)y$. Then, there exists a sequence $\{(x_n,y_n)\}_{n\geq 0}\subset \left(D(A)\cap D(B)\right)\times X$ such that
$$
y_n\to y\quad\text{as}\quad n\to \infty\quad\text{and}\quad \widehat\lambda x_n-\left(A+B\right)x_n = -y_n,\quad\forall n\geq 0.
$$
Next, according to Lemma \ref{LEGraph} $(i2)$, one has
$$
x_n=-\mathcal R\left(\widehat\lambda,A,B\right)y_n,\quad \forall n\geq 0,
$$
so that the sequence $(x_n)_{n\geq 0}$ converges in $X$ toward $x_\infty:=-\mathcal R\left(\widehat\lambda,A,B\right)y\in X$. This ensures that $x=x_\infty\in D\left(L_{\widehat \lambda}\right)$. Since $\overline{A+B}$ is closed, one also deduces that
$$
x_\infty\in D\left(\overline{A+B}\right) \quad \text{and} \quad \widehat\lambda x_\infty-\left(\overline{A+B}\right)x_\infty=-y.
$$
Therefore $x\in D\left(\overline{A+B}\right)$ and $(x,y)\in {\rm Graph}\left(\left(\overline{A+B}\right)-\widehat\lambda\right)$.
Hence $D\left(L_{\widehat{\lambda}}\right)=D\left(\overline{A+B}\right)$ and $L_{\widehat{\lambda}}=\overline{A+B}-\widehat\lambda$. Since $0\in \rho\left(L_{\widehat{\lambda}}\right)$ we obtain that $\widehat\lambda\in \rho\left(\overline{A+B}\right)$ and $\left(\widehat\lambda-\overline{A+B}\right)^{-1}=\mathcal R\left(\widehat \lambda,A,B\right)$.
This completes the proof of the proposition.

\end{proof}

%%%%%%%%%%%%%%%%%%%%%%%%%%%%%%%%%%%%%%%%%%%%%%%%%%%%%%%%%%%%%%%%%%%%%
%																	%
%					Proof of Theorem 3.5							%
%																	%
%%%%%%%%%%%%%%%%%%%%%%%%%%%%%%%%%%%%%%%%%%%%%%%%%%%%%%%%%%%%%%%%%%%%%

We are now in position to complete this section by proving Theorem \ref{THEO_sum}.\\

\begin{proof}[Proof of Theorem \ref{THEO_sum}]

Note that the above proposition already proves the first part of the theorem on the closability of $A+B$.
It also show that when the range of $\lambda-(A+B)$ is dense, for all $\lambda$ large enough, then $\overline{D(\overline{A+B})}=\overline{D(A)\cap D(B)}$ and there exists $\tilde \lambda$ large enough such that $[\tilde \lambda,\infty)\subset \rho\left(\overline{A+B}\right)$ and that
$$
\left(\lambda-\overline{A+B}\right)^{-1}=\mathcal R(\lambda,A,B),\quad\forall \lambda\geq \tilde{\lambda}.
$$
We will now prove ${\rm (ii)}$. To that aim observe that $\{T_{A_0}(t)T_{B_0}(t)\}_{t\geq 0}$ is a strongly continuous semigroup on $\overline{D(A)\cap D(B)}$. Next, for any $x\in \overline{D(A)\cap D(B)}$ and any $\lambda\geq \tilde\lambda$, one has
\begin{equation*}
\begin{split}
\left(\lambda-\left(\overline{A+B}\right)_0\right)^{-1}x &=\mathcal R(\lambda,A,B)x\\
&=\lim_{t\to\infty} \left(S_A\diamond \left(e^{-\lambda(t-\cdot)}T_{B}(t-\cdot)x\right)\right)(t).
\end{split}
\end{equation*}
Since $x\in \overline{D(B)}$, one has $T_B(s)x=T_{B_0}(s)x$, for all $s\geq 0$ while since $x\in \overline{D(A)\cap D(B)}$, one has $T_{B_0}(s)x\in \overline{D(A)}$ for all $s\geq 0$. As a consequence for any $t\geq 0$ we have
\begin{equation*}
\begin{split}
\left(S_A\diamond (e^{-\lambda(t-\cdot)}T_{B}(t-\cdot)x)\right)(t)&=\int_0^t T_{A_0}(t-s)e^{-\lambda(t-s)}T_{B_0}(s)x \d s\\
&=\int_0^t e^{-\lambda s}T_{A_0}(s)T_{B_0}(s)x\d s.
\end{split}
\end{equation*}
Hence we get $x\in \overline{D(A)\cap D(B)}$ and for any $\lambda\geq \tilde\lambda$ we have
\begin{equation*}
\left(\lambda-\left(\overline{A+B}\right)_0\right)^{-1}x=\int_0^\infty e^{-\lambda s}T_{A_0}(s)T_{B_0}(s)x\d s,
\end{equation*}
that ensures that $\left(\overline{A+B}\right)_0$, the part of $\overline{A+B}$ in $\overline{D(A)\cap D(B)}$, is the infinitesimal generator of the strongly continuous semigroup $\left\{T_{A_0}(t)T_{B_0}(t)\right\}_{t\geq 0}$ (see \cite{Engel-Nagel}). This completes the proof of ${\rm (ii)}$. 

We now turn to the proof of ${\rm (iii)}$. To that aim, let us first observe that due to Lemma \ref{LE-estiR}, for all $x \in X$, one has 
$$
\left(\lambda-\overline{A+B}\right)^{-1}x \to 0 \quad\text{as} \quad\lambda\to \infty,
$$
with respect to the operator norm topology in $\mathcal L(X)$. Hence, due to ${\rm (ii)}$, $\overline{A+B}$ satisfies Assumption \ref{ASS1} and thus $\overline{A+B}$ is the infinitesimal generator of a non-degenerate and exponentially bounded integrated semigroup $\{S_{\overline{A+B}}(t)\}_{t\geq 0}$ and for all $\lambda>0$ large enough, one has
\begin{equation*}
\left(\lambda-\overline{A+B}\right)^{-1}x=\lambda\int_0^\infty e^{-\lambda l}S_{\overline{A+B}}(l)x \d l,\quad\forall x\in X.
\end{equation*} 
On the other hand, let us fix $\mu>\omega_A$ and $\lambda\in\R$. Then, we infer from Lemma \ref{LE2.3} that for each $t\geq 0$ one has
\begin{equation*}
(\mu-A_0)^{-1}\left(S_A\diamond \left(e^{-\lambda(t-\cdot)}T_B(t-\cdot)\right)\right)(t)=\int_0^t e^{-\lambda s}T_{A_0}(s)(\mu-A)^{-1}T_B(s)\d s.
\end{equation*}
Hence, using integration by parts we obtain, for all $x\in X$ and $t\geq 0$,
\begin{equation*}
\begin{split}
\left(S_A\diamond \left(e^{-\lambda(t-\cdot)}T_B(t-\cdot)x\right)\right)(t)=\,\,&e^{-\lambda t}\left(S_A\diamond T_B(t-\cdot)x\right)(t)\\
&+\lambda\int_0^t e^{-\lambda s}\left(S_A\diamond T_B(s-\cdot)x\right)(s)\d s.
\end{split}
\end{equation*}
Now, choose $\lambda>0$ large enough such that $\lambda>\eta+\max\left(0,\omega_0(B)\right)$. Then, for all $x \in X$, we have
$$ 
e^{-\lambda t}\left(S_A\diamond T_B(t-\cdot)x\right)(t)\to 0 \quad\text{as}\quad t\to\infty,
$$
and for all $x\in X$ and all $\lambda>0$ large enough,
\begin{equation*}
\begin{split}
\left(\lambda-\overline{A+B}\right)^{-1}x&=\mathcal R(\lambda,A,B)x\\
&=\lim_{t\to\infty}\left(S_A\diamond \left(e^{-\lambda(t-\cdot)}T_B(t-\cdot)x\right)\right)(t)\\
&=\lambda\int_0^\infty e^{-\lambda s}\left(S_A\diamond T_B(s-\cdot)x\right)(s)\d s.
\end{split}
\end{equation*}
This implies that
$$
S_{\overline{A+B}}(t)x=\left(S_A\diamond T_B(t-\cdot)x\right)(t),\quad\forall t\geq 0,\;\forall x\in X,
$$
and this completes the proof of ${\rm (iii)}$.

Finally, let us complete the proof of ${\rm (iv)}$. To do so, let us fix $\tau>0$ and let $f\in C^1\left([0,\tau];X\right)$ be given. Next observe that one has
\begin{equation*}
S_{\overline{A+B}}\ast f(t)=\int_0^t \left(S_A\diamond T_B(l-\cdot)f(t-l)\right)(l)\d l,\;\forall t\in [0,\tau].
\end{equation*}
Next, fix $\mu>\omega_A$ and note that for any $t\in [0,\tau]$ one has
\begin{equation*}
\begin{split}
(\mu-A_0)^{-1}S_{\overline{A+B}}\ast f(t)&=\int_0^t\int_0^l T_{A_0}(s)(\mu-A)^{-1}T_B(s)f(t-l)\d s\d l\\
&=(\mu-A_0)\int_0^t T_{A_0}(s)(\mu-A)^{-1}T_B(s)\hspace{-0.05cm}\left[\int_s^t f(t-l)\d l\right]\hspace{-0.05cm}\d s.
\end{split}
\end{equation*}
Setting $\sigma=t-l$ yields
\begin{equation*}
(\mu-A_0)^{-1}S_{\overline{A+B}}\ast f(t)=\int_0^t T_{A_0}(s)(\mu-A)^{-1}T_B(s)\left[\int_0^{t-s} f(\sigma)\d \sigma\right]\d s,
\end{equation*}
so that, for each $t\in [0,\tau]$, we get
\begin{equation*}
(\mu-A_0)^{-1}\left(S_{\overline{A+B}}\diamond f\right)(t)=\int_0^t T_{A_0}(s)(\mu-A)^{-1}T_B(s)f(t-s)\d s.
\end{equation*}
This ensures that
\begin{equation*}
\left(S_{\overline{A+B}}\diamond f\right)(t)=\Bigl(S_A\diamond \bigl(T_B(t-\cdot)f(\cdot)\bigr)\Bigr)(t),\quad\forall t\geq 0.
\end{equation*}
Now, since for each $t\in (0,\tau]$ the map $s\mapsto T_B(t-s)f(s)\in L^p(0,t;X)$, Assumption \ref{ASS1-A} for the linear operator $A$ ensures that, for all $t\in [0,\tau]$,
\begin{equation*}
\|\left(S_{\overline{A+B}}\diamond f\right)(t)\|\leq M_A \left(\int_0^t e^{p\omega_A(t-s)}\|T_B(t-s)f(s)\|^p \d s\right)^{\frac{1}{p}},\quad\forall t\geq 0.
\end{equation*}
Hence, due to estimate \eqref{esti-TB}, for each $\alpha\in \left(\frac{1}{q^*},\frac{1}{p}\right)$ and each $\omega>\omega_0(B_0)$, there exists some constant $M=M_{\alpha,\omega}$ (independent of $f$) such that for any $t\in [0,\tau]$ and any $1<r<\frac{1}{p\alpha}$ it holds that
\begin{equation*}
\begin{split}
\|\left(S_{\overline{A+B}}\diamond f\right)(t)\|^p&\leq M_A M \int_0^t \frac{1}{(t-s)^{p\alpha}}\,e^{p(\omega_A-\omega)(t-s)}\|f(s)\|^p \d s\\
&\leq M_A M \left(\int_0^t \frac{e^{rp(\omega_A-\omega)(t-s)}}{(t-s)^{rp\alpha}}\d s\right)^{1/r}\left(\int_0^t \|f(s)\|^{r'p} \d s\right)^{1/r'}.\end{split}
\end{equation*}
This prove ${\rm (iv)}$ for any $f$ of class $C^1$ and the result follows from Theorem 3.6 in \cite{Magal-Ruan17}. 
\end{proof}

\section{Applications}

In this section, we study problem \eqref{eq-lin}, that we recall below
\begin{equation}\label{eq-lin'}
\begin{cases}
\left(\partial_t +\partial_a-\Delta_x\right)u(t,a,x)=f(t,a,x), & t>0,\;a>0,\;x\in\Omega,\\
u(t,0,x)=g(t,x), & t>0,\;x\in\Omega,\\
\nabla u(t,a,x)\cdot \nu(x)=h(t,a,x), & t>0,\;a>0,\;x\in\partial\Omega,\\
u(0,a,x) = u_0(a,x), & a>0,\;x \in \Omega.
\end{cases}
\end{equation}
Herein $\Omega\subset \R^N$ denotes a bounded domain with a smooth boundary $\partial\Omega$.

To handle the above problem, fix $q\in (1,\infty)$ and $p\in (1,\infty)$ those specific properties will be discussed below (see condition \eqref{pq}). Our aim in this section is to consider \eqref{eq-lin'} using the results of Theorem \ref{THEO_sum} on the commutative sum of linear operators and to reformulate this partial differential equation as an abstract Cauchy problem involving a non-densely defined operator that generated a suitable integrated semigroup with a well described generator.
Before going to the main result of this section, we need to introduce various spaces and linear operators.

Firstly, we consider the Banach space 
\begin{equation}\label{def-Y}
Y:=W^{1-\frac{1}{q},q}\left(\partial\Omega\right)\times L^q(\Omega),
\end{equation}
as well as the linear operator $B':D(B')\subset Y\to Y$ defined by
\begin{equation*}
\left\lbrace\begin{aligned}
&D(B')=\{0\}\times W^{2,q}(\Omega), \\ 
&B'\begin{pmatrix} 0\\\psi\end{pmatrix}=\begin{pmatrix} -\nabla\psi\cdot\nu\\ \Delta\psi\end{pmatrix}, \quad \forall ~\begin{pmatrix} 0\\\psi\end{pmatrix} \in D(B').
\end{aligned}\right.
\end{equation*}
Here $W^{1-\frac{1}{q},q}\left(\partial\Omega\right)$ denotes the trace space on $\partial \Omega$ associated to functions in $W^{1,q}(\Omega)$.\\Define also the linear operator $B_0':D(B_0')\subset \overline{D(B')}\to \overline{D(B')}$ as the part of $B'$ in $\overline{D(B')}$, that is the linear operator given by
\begin{equation*}
\left\lbrace\begin{aligned}
&D(B'_0)=\left\{ \begin{pmatrix} 0\\ \psi\end{pmatrix} \in \{0\}\times W^{2,q}(\Omega) : \nabla \psi \cdot \nu = 0 \text{ on }\partial\Omega \right\}, \\
&B'_0 \begin{pmatrix} 0\\ \psi\end{pmatrix} = \begin{pmatrix} 0\\ \Delta \psi\end{pmatrix}, \quad \forall ~ \begin{pmatrix} 0\\ \psi\end{pmatrix} \in D(B'_0).
\end{aligned}\right.
\end{equation*}
It is well known (see for instance Henry \cite{Henry} and Yagi \cite{yagi}) that $B'_0$ is a sectorial operator on $\overline{D(B')}=\{0\}\times L^q(\Omega)$.

Next, recalling the definition of the space $Y$ in \eqref{def-Y}, let us define the Banach space
%\begin{equation}\label{def-X}
$$
X=\overline{D(B')}\times L^p(0,\infty; Y)=\left(\{0\}\times L^q(\Omega)\right)\times L^p(0,\infty; Y).
$$
%\end{equation}
We define the linear operator $B:D(B)\subset X\to X$ by
%\begin{equation}\label{B}
$$
\left\lbrace\begin{aligned}
&D(B)=D(B_0')\times L^p(0,\infty,D(B')), \\
&B\begin{pmatrix} y\\ \psi\end{pmatrix}=\begin{pmatrix}B_0'y\\ a\mapsto B'\psi(a,\cdot)\end{pmatrix}, \quad \forall ~\begin{pmatrix} y\\\psi\end{pmatrix} \in D(B),
\end{aligned}\right.
$$
%\end{equation}
as well as $A:D(A)\subset X\to X$ the linear operator given by 
%\begin{equation}\label{A}
$$
\left\lbrace\begin{aligned}
&D(A)=\{0_{\overline{D(B')}}\}\times \left\{\varphi\in W^{1,p}(0,\infty;Y):\;\varphi(0)\in \overline{D(B')}\right\}, \\
&A\begin{pmatrix} 0_{\overline{D(B')}}\\ \varphi\end{pmatrix}=\begin{pmatrix} -\varphi(0,\cdot)\\ -\partial_a \varphi\end{pmatrix}, \quad \forall ~\begin{pmatrix} 0\\ \varphi\end{pmatrix} \in D(A).
\end{aligned}\right.
$$
%\end{equation}

Next we now consider the linear operator $L:D(L)\subset X\to X$ defined as the sum of these two linear operators, that is
%\begin{equation}\label{def-L}
$$
D(L)=D(A)\cap D(B) \quad \text{and} \quad Lx=Ax+Bx, \quad \forall x\in D(A)\cap D(B).
$$
%\end{equation}
In a more concrete form, the operator $L$ expresses as
\begin{equation*}
\left\lbrace\begin{aligned}
&D(L)=D(A)\cap D(B) = \{0_{\overline{D(B')}}\}\times\left( W^{1,p}(0,\infty;Y)\cap  L^p(0,\infty,D(B'))\right), \\
&L\begin{pmatrix} 0_{\overline{D(B')}}\\ \varphi\end{pmatrix} = \begin{pmatrix} -\varphi(0)\\ a\mapsto -\partial_a \varphi(a,\cdot)+B'\varphi(a,\cdot)\end{pmatrix},\quad \forall ~\begin{pmatrix} 0_{\overline{D(B')}}\\ \varphi\end{pmatrix} \in D(L).
\end{aligned}\right.
\end{equation*}
The main result of this section is concerned with the closability of the linear operator $L$ and with property of the integrated semigroup generated by $L$.
Our result reads as follows.
\begin{theorem}\label{THEO-L}
Assume that the parameters $p>1$ and $q>1$ satisfy the following condition\begin{equation}\label{pqs}
1<\frac{1+q}{2q}+\frac{1}{p}.
\end{equation}
Then the linear operator $L:D(L)\subset X\to X$ is closable. Its closure $\overline{L}:D\left(\overline{L}\right)\subset X\to X$ enjoys the following properties
\begin{itemize}
\item [{\rm (i)}] $D(A)\cap D(B)\subset D(L)$ and $\overline{D\left(\overline{L}\right)}=\overline{D(A)\cap D(B)}$,
\item[{\rm (ii)}] The linear operator $\overline{L}:D(\overline{L})\subset X\to X$ satisfies Assumption \ref{ASS1-A} with parameter $p$ replaced by any $r\in (1,\infty)$ such that
\begin{equation*}
1+\frac{1}{r}<\frac{1+q}{2q}+\frac{1}{p}.
\end{equation*} 
\end{itemize}
\end{theorem}
To prove this result we apply Theorem \ref{THEO_sum} with the linear operators $A$ and $B$ defined above. To that aim, let us first recall some properties satisfied by these operators.

We start with the description of the linear operator $A:D(A)\subset X\to X$. It enjoys the following properties.
\begin{lemma}\label{LE-opA}
Let $p>1$ be given. The linear operator $A:D(A)\subset X\to X$ satisfies Assumption \ref{ASS1-A}.
More precisely, one has
\begin{itemize}
\item [{\rm (i)}] $(0,\infty)\subset \rho(A)$ and for all $\lambda>0$ one has
\begin{equation*}
(\lambda-A)^{-1}\begin{pmatrix} y\\\psi\end{pmatrix} =\begin{pmatrix} 0\\ e^{-\lambda \cdot} y+\int_0^\cdot e^{-\lambda(\cdot-s)}\psi(s)\d s\end{pmatrix},\quad\forall \begin{pmatrix} y\\\psi\end{pmatrix}\in X,
\end{equation*}
and
\begin{equation*}
\left\|(\lambda-A)^{-1}\right\|_{\mathcal L(X)}=\mathcal O\left(\lambda^{-\frac{1}{p}}\right)\quad\text{as} \quad\lambda\to \infty.
\end{equation*}
\item[{\rm (ii)}] The part of $A$ in $\overline{D(A)}$, denoted by $A_0:D(A_0)\subset \overline{D(A)}\to \overline{D(A)}$ is the infinitesimal generator of a strongly continuous semigroup, denoted by $\left\{T_{A_0}(t)\right\}_{t\geq 0}$, on $\overline{D(A)}$.
\item[{\rm (iii)}] The integrated semigroup $\left\{S_A(t)\right\}_{t\geq 0}\subset \mathcal L(X)$ generated by $A:D(A)\subset X\to X$ is given, for each $\begin{pmatrix} g\\ h\end{pmatrix}\in L_{\rm loc}^p\left([0,\infty);X\right)$, by
\begin{equation*}
\left(S_A\diamond \begin{pmatrix} g\\ h\end{pmatrix}\right)(t)=\begin{pmatrix} 0\\ g(t-\cdot)\chi_{(0,t)}(\cdot)\end{pmatrix}+\int_0^t T_{A_0}(t-s)\begin{pmatrix} 0\\ h(s)\end{pmatrix}\d s,
\end{equation*}
wherein we have used $\chi$ to denote the characteristic function of a set.
\end{itemize}
\end{lemma}

The proof of this lemma can be found in \cite{Magal-Ruan07, Magal-Ruan17}.

We are now concerned with the linear operator $B':D(B')\subset Y\to Y$ that enjoys the following properties.
\begin{lemma}\label{LE-opBp}
Let $q>1$ be given. Then, the linear operator $B':D(B')\subset Y\to Y$ satisfies Assumption \ref{ASS1-B} with $p^*=\frac{2q}{1+q}$. One furthermore has $(0,\infty)\subset \rho\left(B'\right)=\rho\left(B_0'\right)$. \end{lemma}
The proof of this lemma follows from the elliptic estimates derived by Agranovich et al. \cite{agranovich}.

Using this property we are now able to discuss some properties satisfied by the linear operator $B:D(B)\subset X\to X$.

\begin{lemma}\label{LE-opB}
Let $q>1$ be given.
Then the linear operator $B:D(B)\subset X\to X$ satisfies Assumption \ref{ASS1-B} with $p^*=\frac{2q}{1+q}$. Moreover one has
\begin{equation*}
(0,\infty)\subset \rho(B)=\rho(B')=\rho(B'_0),
\end{equation*}
and for all $\lambda\in \rho(B)$,
\begin{equation*}
(\lambda-B)^{-1}\begin{pmatrix} y\\ \psi\end{pmatrix}=\begin{pmatrix} (\lambda-B_0')^{-1}(y)\\ a\mapsto (\lambda-B')^{-1}\psi(a,\cdot)\end{pmatrix}.
\end{equation*}
\end{lemma}
The proof of this lemma directly follows from the properties of the linear operator $B'$ stated in Lemma \ref{LE-opBp} above.

Using the above properties, we are able to prove Theorem~\ref{THEO-L}. As already mentioned above, its proof is a direct application of Theorem~\ref{THEO_sum}.

\begin{proof}
First note that due to Lemma~\ref{LE-opA} and Lemma~\ref{LE-opB}, to apply Theorem~\ref{THEO_sum}, it is sufficient to check that $A$ and $B$ are resolvent commuting and that ${\rm R}\left(\lambda-\left(A+B\right)\right)$ is dense in $X$ for all $\lambda$ large enough.

\noindent{\bf Resolvent commutativity:} This property follows from the resolvent formula described in Lemma \ref{LE-opA} and \ref{LE-opB}.
Indeed, let $\lambda\in \rho(A)$ and $\mu\in \rho(B)$ be given. Then, for each $(y,\psi)^T\in X$ one has
\begin{equation*} 
\begin{cases}
(\mu-B)^{-1}(\lambda-A)^{-1}\begin{pmatrix} y\\ \psi\end{pmatrix}=(\mu-B)^{-1}\begin{pmatrix} 0\\ e^{-\lambda \cdot} y+\int_0^\cdot e^{-\lambda(\cdot-s)}\psi(s)\d s\end{pmatrix},\quad\forall \begin{pmatrix} y\\\psi\end{pmatrix},\\
\begin{pmatrix} 0\\ e^{-\lambda \cdot} (\mu-B_0')^{-1}y+\int_0^\cdot e^{-\lambda(\cdot-s)}(\mu-B')^{-1}\psi(s)\d s\end{pmatrix},\quad\forall \begin{pmatrix} y\\\psi\end{pmatrix},
\end{cases}
\end{equation*}
while
\begin{equation*} 
\begin{cases}
(\lambda-A)^{-1}(\mu-B)^{-1}\begin{pmatrix} y\\ \psi\end{pmatrix}=(\lambda-A)^{-1}\begin{pmatrix} (\mu-B_0')^{-1}(y)\\ a\mapsto (\mu-B')^{-1}\psi(a,\cdot)\end{pmatrix}, \\
=\begin{pmatrix} 0\\ e^{-\lambda \cdot} (\mu-B_0')^{-1}y+\int_0^\cdot e^{-\lambda(\cdot-s)} (\mu-B')^{-1}\psi(s)\d s\end{pmatrix}.
\end{cases}
\end{equation*}
Hence $A$ and $B$ are resolvent commuting.

\noindent{{\bf Density of the range:}} To prove that ${\rm Ran\,(\lambda-A-B)}$ is dense for all $\lambda$ large enough, let us notice that one has
\begin{equation*}
\begin{split}
&\overline{D(A)}=\left\{0_{\overline{D(B')}}\right\}\times L^p\left(0,\infty;Y\right),\\
&\overline{D(B)}=\overline{D(B')}\times L^p\left(0,\infty;\overline{D(B')}\right),
\end{split}
\end{equation*}
Hence $\overline{D(A)}+\overline{D(B)}=X$. Hence Property \eqref{pq} in Theorem \ref{THEO_sum} applies and ensures that the range of $\lambda-A-B$ is dense in $X$ for all $\lambda$ large enough. Hence Condition \eqref{cond} in Theorem \ref{THEO_sum} is satisfied.

As a conclusion, since $p^*=\frac{2q}{1+q}$, Condition \eqref{pqs} ensures that Assumption \ref{ASS1-C} is satisfied. Applying Theorem \ref{THEO_sum} with the linear operators $A$ and $B$ completes the proof of Theorem \ref{THEO-L}.

\end{proof}

Now let us come back to \eqref{eq-lin'}. Let $p>1$ and $q>1$ be given satisfying \eqref{pqs}. Let $r>1$ such that
\begin{equation*}
1+\frac{1}{r}<\frac{1+q}{2q}+\frac{1}{p}.
\end{equation*} 
Now let $T>0$ be given and fixed. Consider $u_0=u_0(a,x)\in L^p\left(0,\infty;L^q(\Omega)\right)$, $f=f(t,a,x)\in L^r\left(0,T;L^p\left(0,\infty;L^q(\Omega)\right)\right)$, $g=g(t,x)\in L^r\left(0,T;L^q(\Omega)\right)$ and $h=h(t,a,x)\in L^r\left(0,T; L^p\left(0,\infty; W^{1-\frac{1}{q},q}(\partial\Omega)\right)\right)$.
Next set $F\in L^r(0,T;X)$ the function defined by
$$
F(t)=\begin{pmatrix} \begin{pmatrix} 0\\ g(t,\cdot)\end{pmatrix}\\a\mapsto \begin{pmatrix} h(t,a,\cdot)\\ f(t,a,\cdot)\end{pmatrix}\end{pmatrix}\in L^r\Bigl(0,T;\bigl(\overline{D(B')}\times L^p(0,\infty;Y)\bigr)\Bigr).
$$
Then problem \eqref{eq-lin'} formally rewrites as
\begin{equation*}
\left\lbrace\begin{aligned}
&\frac{\d v(t)}{\d t}=Lv(t)+F(t),\quad t\in (0,T), \\
&v(0) = v_0.
\end{aligned}\right.
\end{equation*}
wherein we have identified $v(t)\in \overline{D(L)}=\overline{D(A)\cap D(B)}$ together with the map 
$$t\mapsto \begin{pmatrix}0_{\overline{D(B')}}\\  a\mapsto \begin{pmatrix}0_{W^{1-\frac{1}{q},q}\left(\partial\Omega\right)}\\ u(t,a,\cdot)\end{pmatrix}\end{pmatrix}\in\overline{D(A)\cap D(B)}\subset X,$$
while $v(0)=v_0$ is defined similarly with $u(t,a,\cdot)$ replaced by $u_0(a,\cdot)$.\\
According to Theorem \ref{THEO-L}, $L$ is closable, so that we consider the closure equation
\begin{equation}\label{cauchy}
\left\lbrace\begin{aligned}
&\frac{\d v(t)}{\d t}=\overline{L}v(t)+F(t),\quad t\in (0,T), \\
&v(0) = v_0.
\end{aligned}\right.
\end{equation}
Herein $\overline L$ denotes the closure of $L$. Next using the properties of the linear operator $\overline L$, the above equation has a unique mild solution $u\in C\left([0,T]; \overline{D(A)\cap D(B)}\right)$ given by
\begin{equation*}
v(t)=T_{\overline L_0}(t)v_0+\left(S_{\overline L}\diamond F\right)(t),\quad \forall t\in [0,T],
\end{equation*}
wherein $\left\{T_{\overline L_0}(t)\right\}_{t\geq 0}\subset \mathcal L\left(\overline{D(A)\cap D(B)}\right)$ denotes the strongly continuous semigroup generated by $\overline L_0$, the part of $\overline{L}$ in $\overline{D(A)\cap D(B)}$ while $\left\{S_{\overline L}(t)\right\}_{t\geq 0}\subset \mathcal L(X;\overline{D(A)\cap D(B)})$ denotes the integrated semigroup on $X$ generated by $\overline L$.

Now using the formula for $T_{\overline L_0}$ and $S_{\overline L}$ expressed in our general result Theorem \ref{THEO_sum} in term of $T_{A_0}$ and $T_{B_0}$, that respectively denote the strongly continuous semigroups generated by $A_0$ and $B_0$, the parts of $A$ and $B$ in $\overline{D(A)}$ and $\overline{D(B)}$; and $S_A$, $S_B$ the integrated semigroups generated by $A$ and $B$ respectively, one may obtain the formula below for the mild solution $u=u(t,a,x)$ of \eqref{eq-lin'}.
\begin{equation*}
\begin{split}
v(t)=&T_{\overline L_0}(t)v_0+\left(S_{\overline L}\diamond F\right)(t)\\
=&T_{A_0}(t)T_{B_0}(t)v_0+\left(S_A\diamond \left(T_B(t-\cdot)F(\cdot)\right)\right)(t).
\end{split}
\end{equation*} 
Recalling that $\left\{S_{B'}(t)\right\}_{t\geq 0}\subset \mathcal L(Y)$ denotes the integrated semigroup generated by $B'$ and $\left\{T_{B_0'}(t)\right\}_{t\geq 0}\subset \mathcal L\left(\overline{D(B')}\right)$ is the strongly continuous semigroup generated by $B_0'$, the part of $B'$ in $\overline{D(B')}$ then one obtains the following explicit formula for the function $u=u(t,a,x)\in C\left([0,T]; L^p\left(0,\infty;L^q(\Omega)\right)\right)$, the mild solution of \eqref{eq-lin'}.
\begin{lemma}[Mild solution]\label{LEmild}
The mild solution of the abstract Cauchy problem \eqref{cauchy} (which corresponds to \eqref{eq-lin'}) is given -- in $Y$ -- for $t\in [0,T]$ and almost every $a>0$  by 
\begin{equation*}
\begin{pmatrix}
0_{W^{1-\frac{1}{q},q}(\partial\Omega)}\\
u(t,a,\cdot)
\end{pmatrix}
=\left\lbrace
\begin{array}{ll}
T_{B_0'}(t)\begin{pmatrix}
0\\ u_0(a-t,\cdot)
\end{pmatrix}
+\left(S_{B'} \diamond  \begin{pmatrix}
h(\cdot, \cdot+a-t)\\ f(\cdot,\cdot+a-t) 
\end{pmatrix} \right)(t), & \text{if } t<a, \medskip \\ 
T_{B_0'}(a)\begin{pmatrix}
0\\ g(t-a,\cdot)
\end{pmatrix}
\,\,+\left(S_{B'} \diamond  \begin{pmatrix}
h(\cdot, \cdot+t-a)\\ f(\cdot,\cdot+t-a) 
\end{pmatrix} \right)(a), & \text{if } t>a.
\end{array}
\right.
\end{equation*}
\end{lemma}
The above formula for the weak solution of the abstract Cauchy problem \eqref{cauchy} corresponds to the formal expression of the solution of \eqref{eq-lin'} that can be obtained using integration along the characteristics and also to the decomposition \eqref{eq-lin1}-\eqref{eq-lin2}-\eqref{eq-lin3} proposed in the introduction.


\begin{thebibliography}{99}

\bibitem{agranovich} M. S. Agranovich, R. Denk and M. Faierman, Weakly smooth nonselfadjoint spectral elliptic boundary problems, \textit{Mathematical topics}, \textbf{14} (1997), 138-199.

\bibitem{Arendt1} W. Arendt, Resolvent positive operators, \textit{Proc. London Math. Soc.}, \textbf{54 }(1987), 321-349.

\bibitem{Arendt2}  W. Arendt, Vector valued Laplace transforms and Cauchy problems, \textit{Israel J. Math.}, \textbf{59} (1987), 327-352.

\bibitem{Arendt-book} W. Arendt, C. J. K. Batty, M. Hieber, and F. Neubrander, \textit{Vector-Valued Laplace Transforms and Cauchy Problems}, Birkh\"auser, Basel, 2001.

\bibitem{arendt-bu-haase} W.~Arendt, S.~Bu and M.~Haase, Functional calculus, variational methods and Liapunov's theorem, {\it Arch. Math.}, \textbf{77} (2001), 65-75.

\bibitem{Daprato-Grisvard}G.\ Da Prato and P.\ Grisvard, Somme d'op\'{e}rateurs lin\'{e}aires et \'{e}quations diff\'{e}rentielles op\'{e}rationnelles, \textit{J. Math. Pures Appl.} \textbf{54} (1975), 305-387.

\bibitem{DaPrato-Sinestrari} G. Da Prato, E. Sinestrari, Differential operators with non dense domain,\textit{ Annali della Scuola Normale Superiore di Pisa-Classe di Scienze}, \textbf{14(2)} (1987), 285-344.

\bibitem{DiBlasio} G. Di Blasio, Non-linear age-dependent population diffusion, \textit{Journal of Mathematical Biology}, \textbf{8} (1979), 265-284.

\bibitem{dore-venni} G. Dore and A. Venni, On the closedness of the sum of two closed operators, \textit{Mathematische Zeitschrift}, \textbf{196} (1987), 189-201.

%\bibitem{Ducrot} A. Ducrot, Travelling wave solutions for a scalar age-structured equation, \textit{DCDS B}, \textbf{7} (2007), 251-273.

\bibitem{Ducrot-Magal} A. Ducrot, P. Magal, Travelling wave solutions for an infection-age structured model with diffusion, \textit{Proceedings of the Royal Society of
Edinburgh: Section A Mathematics}, \textbf{139} (2009), 459-482.

\bibitem{DM} A. Ducrot and P. Magal, Travelling wave solution for infection-age structured model with vital dynamics, \textit{Nonlinearity} \textbf{24} (2011), 2891-2911. 

\bibitem{DM-Pisa} A. Ducrot and P. Magal, Integrated Semigroups and Parabolic Equations. Part II: Semilinear problems, \textit{Annali della Scuola Normale Superiore Di Pisa, Classe di Scienze}, to appear.

\bibitem{DM-TAMS} A. Ducrot and P. Magal, A center manifold for second order semi-linear differential equations on the real line and applications to the existence of wave trains for the Gurtin-McCamy equation, \textit{Trans. Amer. Math. Soc.}, \textbf{372} (2019), 3487-3537.

\bibitem{Ducrot-Magal-Prevost09} A. Ducrot, P. Magal and K. Prevost,
Integrated Semigroups and Parabolic Equations. Part I: Linear Perburbation
of Almost Sectorial Operators, {\it J. Evol. Equ.}, \textbf{10} (2010), 263-291.

\bibitem{Engel-Nagel} K.-J. Engel and R.\ Nagel, \textit{One Parameter
Semigroups for Linear Evolution Equations},\ Springer-Verlag, New York, 2000.

\bibitem{Gurtin1} M.E.  Gurtin  and  R.C.  MacCamy,  Non-linear  age-dependent  population dynamics, \textit{Arch. Rational Mech. Anal.}, \textbf{54} (1974), 281-300.

\bibitem{Gurtin2} M.E. Gurtin and R.C. MacCamy, On the diffusion of biological popula-tions, \textit{Math. Biosci.}, \textbf{33} (1977), 35-49.

\bibitem{Henry} D. Henry, \textit{Geometric theory of semilinear parabolic
equations}, Lecture Notes in Mathematics, vol. 840 Springer-Verlag (1981).

%\bibitem{Kalton} N. Kalton, L. Weis, The $H^\infty-$calculus and sums of closed operators, \textit{Math. Ann.}, \textbf{321} (2001), 319-345. 

\bibitem{Kellermann} H. Kellermann and M. Hieber, Integrated semigroups, \textit{J.\ Funct. Anal.}, \textbf{84} (1989), 160-180.

\bibitem{Kubo-Langlais} M. Kubo and M. Langlais, Periodic solutions for nonlinear population dynamics models with age-dependence and spatial structure,\textit{ J. Differential Equations}, \textbf{109} (1994), 274-294.

\bibitem{labbas} R. Labbas, Some results on the sum of linear operator with nondense domains, \textit{Annali di Matematica Pura ed Applicata}, \textbf{154.1 }(1989), 91-97.

\bibitem{Langlais} M. Langlais, Large time behavior in a nonlinear age-dependent population dynamics problem with spatial diffusion, \textit{J. Math. Biol.}, \textbf{26} (1988), 319-346.

\bibitem{LMR12} Z. Liu, P. Magal and S. Ruan, Center-unstable manifold theorem for non-densely defined Cauchy problems, and the stability of bifurcation periodic orbits by Hopf bifurcation, \textit{Canadian Applied Mathematics Quarterly,} \textbf{2} (2012), 135-178. 

\bibitem{Lunardi} A. Lunardi, \textit{Analytic semigroups and optimal
regularity in parabolic problems}, Birkhauser, Basel, 1995.

\bibitem{Magal-Ruan07} P. Magal, and S. Ruan, On Integrated Semigroups and
Age Structured Models in $L^{p}$ Spaces, \textit{Differential and Integral
Equations}, \textbf{20} (2007), 197-139.

\bibitem{Magal-Ruan09a} P. Magal and S. Ruan, Center Manifolds for
Semilinear Equations with Non-dense Domain and Applications to Hopf
Bifurcation in Age Structured Models, \textit{Memoirs of the American
Mathematical Society} \textbf{202} (2009), no. 951.

\bibitem{Magal-Ruan09b} P. Magal and S. Ruan, On Semilinear Cauchy Problems
with Non-dense Domain, \textit{Advances in Differential Equations}, \textbf{14} (2009), 1041-1084.

\bibitem{Magal-Ruan17} P. Magal and S. Ruan, \textit{Theory and Applications of
Abstract Semilinear Cauchy Problems}, Springer-Verlag, 2018

\bibitem{MLR14} Z. Liu, P. Magal and S. Ruan, Normal forms for semilinear equations with non-dense domain with applications to age structured models, \textit{J. Differential Equations}, \textbf{257 }(2014) 921-1011.

%\bibitem{monniaux} S. Monniaux, A perturbation result for bounded imaginary powers, \textit{Archiv der Mathematik}, \textbf{68} (1997), 407-417.

\bibitem{Neubrander} F. Neubrander, Integrated semigroups and their application to the abstract Cauchy problem, \textit{Pac.\ J.\ Math.}, \textbf{135} (1988), 111-155.

\bibitem{periago} F. Periago and B. Straub, A functional calculus for almost
sectorial operators and applications to abstract evolution equations, 
{\it J. Evol. Equ.}, \textbf{2} (2002), 41-68.

\bibitem{pruss-sohr}  J.~Pr\"uss and H.~Sohr, On operators with bounded imaginary powers in Banach spaces, \textit{Mathematische Zeitschrift}, \textbf{203} (1990), 429-452.

%\bibitem{pruss-sohr2} J.~Pr\"uss and H.~Sohr, Imaginary powers of elliptic second order differential operators in $L^p$-spaces, \textit{Hiroshima Math. J.}, \textbf{23} (1993), 161-192. 

\bibitem{roidos} N. Roidos, Closedness and invertibility for the sum of two closed operators, \textit{Advances in Operator Theory}, \textbf{3} (2016), 582-605. 

\bibitem{roidos2} N. Roidos, On the inverse of the sum of two sectorial operators, \textit{Journal of Functional Analysis}, \textbf{265} (2013), 208-222.

\bibitem{Thieme97} H. R.\ Thieme, On commutative sums of generators, \textit{Rendiconti Instit. Mat. Univ. Trieste}, \textbf{28} (1997), Suppl., 421-451.

\bibitem{Thieme90a} H. R. Thieme, \textquotedblleft Integrated
semigroups\textquotedblright\ and integrated solutions to abstract Cauchy problems, \textit{J. Math. Anal. Appl.}, \textbf{152 }(1990), 416-447.

\bibitem{Thieme08} H. R. Thieme, Differentiability of convolutions, integrated semigroups of bounded semi-variation and the inhomogeneous Cauchy problem, \textit{J. Evol.
Equ.}, \textbf{8} (2008), 283-305.

%\bibitem{Walker} C. Walker, Positive equilibrium solutions for age-and spatially-structured population models,  \textit{SIAM J. Math. Anal.}, \textbf{41} (2009), 1366-1387.

\bibitem{Walker1} C. Walker, Some remarks on the asymptotic behavior of the semigroup associated with age-structured diffusive populations, \textit{Monatshefte f\"ur Mathematik}, \textbf{170} (2012), 481-501.

\bibitem{Webb08} G.F. Webb, \textit{Population models structured by age, size, and spatial position}, In: Magal P., Ruan S. (eds) Structured Population Models in Biology and Epidemiology. Lecture Notes in Mathematics, vol \textbf{1936}. Springer Berlin Heidelberg (2008). 

\bibitem{yagi} A. Yagi, \textit{Abstract parabolic evolution equations and their applications}, Springer Science \& Business Media, 2009.

\end{thebibliography}
\end{document}